\documentclass[numbers,webpdf,imanum]{ima-authoring-template}
\usepackage{amsfonts}
\usepackage{amsmath}
\usepackage{amssymb}
\usepackage{array}
\usepackage{graphicx}
\usepackage{color}
\usepackage{stmaryrd}
\def\l2{_{L_2(\Omega)}}
\def\nl2{_{[L_2(\Omega)]^n}}
\def\E{\mathcal{E}}

\def\Div{{\rm div}\hspace{1pt}}

\newcommand{\va}{{\mathbf a}}
\newcommand{\II}{{I\!I}}

 \def\XXint#1#2#3{{\setbox0=\hbox{$#1{#2#3}{\int}$}
\vcenter{\hbox{$#2#3$}}\kern-.5\wd0}}

\DeclareFontFamily{U}{matha}{\hyphenchar\font45}
\DeclareFontShape{U}{matha}{m}{n}{
      <5> <6> <7> <8> <9> <10> gen * matha
      <10.95> matha10 <12> <14.4> <17.28> <20.74> <24.88> matha12
      }{}
\DeclareSymbolFont{matha}{U}{matha}{m}{n}
\DeclareFontSubstitution{U}{matha}{m}{n}

\DeclareFontFamily{U}{mathx}{\hyphenchar\font45}
\DeclareFontShape{U}{mathx}{m}{n}{
      <5> <6> <7> <8> <9> <10>
      <10.95> <12> <14.4> <17.28> <20.74> <24.88>
      mathx10
      }{}
\DeclareSymbolFont{mathx}{U}{mathx}{m}{n}
\DeclareFontSubstitution{U}{mathx}{m}{n}

\DeclareMathDelimiter{\vvvert}{0}{matha}{"7E}{mathx}{"17}

\def\T{\mathcal{T}}

\def\BB{{\mathcal B}}

\def\E{\mathcal{E}}

\definecolor{darkred}{rgb}{.7,0,0}

\definecolor{green}{rgb}{0,0.7,0}

\definecolor{myblue}{rgb}{0,0,0.7}%{0,0,0}%
\newcommand\Blue[1]{\textcolor{myblue}{#1}}

\definecolor{pass}{rgb}{0,0,0.5}

\theoremstyle{thmstyletwo}%
\newtheorem{theorem}{Theorem}%  meant for continuous numbers
\newtheorem{remark}{Remark}%
\newtheorem{lemma}{Lemma}
\newtheorem{corollary}{Corollary}

\numberwithin{equation}{section}

\newcounter{saveeqn}

   \usepackage{color}  %!!!!!!!!!!!!!
     \definecolor{red}{rgb}{0.9,0,0}
     \definecolor{green}{rgb}{0,0.6,0}
     \definecolor{rb}{rgb}{0.6,0,0.2}     %===provisionally ok
     \definecolor{nat}{rgb}{0,0,0.5}%{0,0,0}%
\newcommand{\pt}{\partial}
\newcommand {\eps} {\varepsilon}

\newcommand {\beq} {\begin{equation}}
\newcommand {\eeq} {\end{equation}}

\newcommand {\beqa} {\begin{eqnarray}}
\newcommand {\eeqa} {\end{eqnarray}}
\newcommand {\beqann} {\begin{eqnarray*}}
\newcommand {\eeqann} {\end{eqnarray*}}

\renewcommand{\epsilon}{\varepsilon}
\numberwithin{equation}{section}

\begin{document}

\DOI{DOI HERE}
\copyrightyear{2021}
\vol{00}
\pubyear{2021}
\access{Advance Access Publication Date: Day Month Year}
\appnotes{Paper}
\copyrightstatement{Published by Oxford University Press on behalf of the Institute of Mathematics and its Applications. All rights reserved.}
\firstpage{1}

\title[Pointwise estimators for convection-diffusion problems]{Maximum norm a posteriori error estimates for convection-diffusion problems}

\author{Alan Demlow
\address{\orgdiv{Department of Mathematics}, \orgname{Texas A\&M University}, \orgaddress{\street{College Station}, \postcode{77843-3368}, \state{TX}, \country{USA}}}}
\author{Sebastian Franz
\address{\orgdiv{Institute of Scientific Computing}, \orgname{Technische Universit\"at Dresden}, \orgaddress{%\street{Street}, \postcode{Postcode},
 \state{Dresden}, \country{Germany}}}}
\author{Natalia Kopteva*
\address{\orgdiv{Department of Mathematics and Statistics}, \orgname{University of Limerick}, \orgaddress{%\street{Street}, \postcode{Postcode},
\state{Limerick}, \country{Ireland}}}}

%\orgname{Department of Mathematics, Texas A\&M University, College Station TX, 77843; email: {\tt demlow@math.tamu.edu}.
%Partially supported by NSF Grants DMS-1720369 and DMS-2012326.
%}
%\and
%Sebastian Franz\thanks{Institute of Scientific Computing, Technische Universit\"at Dresden, Germany;
%          email: {\tt sebastian.franz@tu-dresden.de}.}
%\and
%Natalia Kopteva\thanks{Department of Mathematics and Statistics, University of Limerick, Limerick, Ireland; email: {\tt natalia.kopteva@ul.ie}.
%Partially supported by  Science Foundation Ireland under Grant number 18/CRT/6049.}
%}

\authormark{A.~DEMLOW, S. FRANZ, AND N.~KOPTEVA}

\corresp[*]{Corresponding author: \href{email:Natalia.Kopteva@ul.ie}{Natalia.Kopteva@ul.ie}}

\received{Date}{0}{Year}
\revised{Date}{0}{Year}
\accepted{Date}{0}{Year}

\renewcommand{\thefootnote}{\fnsymbol{footnote}}

\abstract{
We prove residual-type a posteriori error estimates in the maximum norm for a linear scalar elliptic convection-diffusion problem that may be singularly perturbed.  Similar error analysis in the energy norm by Verf\"{u}rth indicates that a dual norm of the {convective derivative of the} error must be added to the natural energy norm in order for the natural residual estimator to be reliable and efficient.  We show that  the situation is similar for the maximum norm.   In particular, we define a mesh-dependent weighted seminorm of the convective  error which functions as a maximum-norm counterpart to the dual norm used in the energy norm setting.  The total error is then defined as the sum of this seminorm, the maximum norm of the error, and data oscillation.  The natural maximum norm residual error estimator is shown to be equivalent to this total error notion, with constant independent of singular perturbation parameters.  These estimates are proved under the assumption that certain natural estimates hold for the Green's function for the problem at hand.  Numerical experiments confirm that our estimators effectively capture the maximum-norm error behavior for singularly perturbed problems, and can effectively drive adaptive refinement in order to capture layer phenomena.}


%\begin{AM} 65N12, 65N15,  65N30, 65N50  	
% 65N50  	Mesh generation, refinement, and adaptive methods for boundary value problems involving PDEs
%\end{AM}

%\pagestyle{myheadings}
%\thispagestyle{plain}

\maketitle

%------------------------------------------------------------------------------------------------------------------------------------------
\section{Introduction}

Our goal is to prove residual-type a posteriori error estimates in the maximum norm for singularly perturbed convection-diffusion equations of the form
%\footnote{I had to replace the definition of '\{$\backslash$va\}' from '{$\backslash$boldsymbol\{a\}}' (appearing as
%a regular $\boldsymbol{a}$) to '{\{$\backslash$mathbf a\}}'
%(appearing as $\mathbf a$). Maybe you can find another solution.
%I think IMA style somehow suppresses the bold font in formulas?|}
\begin{align}
\label{eq1-1}
 Lu:=-\eps\Delta u+\Div(\va u)+bu=f\;\;\hbox{in } \Omega, \qquad u=0\;\; \hbox{on } \partial \Omega.
\end{align}
Here  $0<\epsilon \le 1$,
$\Omega$ is a
polyhedral domain in $\mathbb{R}^n$, $n=2, 3$,
and we assume that $\va=(a_1,\ldots,a_n)$, $b$, and $f $ are sufficiently smooth on $\bar\Omega $, and that
$|\va|>0$, {$ b\ge 0$, and
 $b+\frac{1}{2}\Div \va\geq 0$ in $\bar\Omega$.}
 We make additional assumptions on the Green's function of $L$, which agree with the sharp bounds on the Green's function for a particular case of \eqref{eq1-1}, rigorously proved in \cite{FrKop2d, FrKop3d, FrKop_sharp,FKPP}.  We emphasize that our estimates are robust with respect to the singular perturbation parameter $\epsilon$, up to logarithmic terms that typically arise in the context of maximum norm estimates for finite element methods.

A guiding principle of a posteriori error estimation is that estimators should be reliable and efficient ({i.e.} provide a posteriori upper and lower bounds) for the error notion under consideration.  For symmetric elliptic problems, residual-type error estimators of the type we consider here are well known to be reliable and efficient for energy and a number of $L_p$-type norms, up to a data oscillation term which is heuristically of higher order and measures the distance of the right hand side $f$ to a piecewise polynomial space.
By contrast, for convection-diffusion problems {it is} known that standard residual estimators reliably bound the \Blue{error in the energy norm}
{defined by $\vvvert v \vvvert^2 = \epsilon \|\nabla v\|_{L_2(\Omega)}^2+ \|(b+\frac{1}{2}\Div \va)^{1/2}v\|_{L_2(\Omega)}^2$}, but are not efficient.
To be more precise {\cite{Ver_sinum_2005},
the natural energy-norm residual estimator then is reliable and efficient up to data oscillation for the error notion $\vvvert u-u_h \vvvert + \vvvert \va \cdot \nabla(u-u_h) \vvvert_\ast$,
where $\vvvert \varphi \vvvert_\ast: = \sup_{ v \in H_0^1(\Omega) \setminus \{ 0 \}} \frac{(\varphi, v)}{\vvvert v \vvvert}$,
that is for the sum of the energy norm and a dual norm of the {convective derivative of the error (also to be referred to as the convective error).}
  The dual convective norm is weaker and thus may be asymptotically negligible, but can also dominate the error as measured in the energy norm until layers in the solution are sufficiently refined.  This framework is explained in \cite{tob_Ver_2015}, where residual-type a posteriori error estimates are proved for several stabilized finite element methods for equations of the type \eqref{eq1-1} {(see also \cite{Ver_sinum_2005,Ver_book_13})}.  Philosophically we closely follow this work, but while considering the maximum norm rather than an energy norm.

We begin by proving residual-type estimates in the maximum norm for standard Galerkin finite element methods without stabilization.  Such methods are not necessarily practically relevant in the
%case of singularly perturbed problems
{convection-dominated regime,}
but will allow us to establish a suitable theoretical framework before adding stabilization terms.
{Next, in order to obtain a reliable and efficient estimator,
we add an
elementwise-weighted measure of the convective derivative of the error to our error notion.}
  This seminorm of the convective error has a similar purpose to the one defined for the energy norm in \cite{tob_Ver_2015}, {and is similarly independent of the stabilization terms}.  However, in contrast to that work, our seminorm is mesh-dependent and is adapted to the particular case of the maximum norm.  A precise definition is given below, but this seminorm behaves similarly to the quantity $\|\min\{1,\,\ell_{h} \eps^{-1}h_T^2\} \,\va\cdot\nabla (u-u_h)\|_{L_\infty(\Omega)}$ {under sufficiently restrictive assumptions}.  Here $u_h$ is the finite element solution, $\ell_h$ is a logarithmic factor \Blue{depending on $\epsilon$ and the minimum mesh diameter}, and $h_T$ is the local mesh size.  This quantity typically dominates the original target error notion $\|u-u_h\|_{L_\infty(\Omega)}$ when $h_T>\epsilon$, but becomes relatively negligible when $h_T \ll\epsilon$.  This error structure is similar to that observed in the energy norm case.
  After considering unstabilized finite element methods we consider the effects of several stabilization schemes on our a posteriori estimates, also following the similar analysis for energy norms outlined in \cite{tob_Ver_2015}.

As in recent works on maximum-norm a posteriori error estimation \cite{DG12, DK15}, we shall rely on the continuous Green's function for the adjoint problem to \eqref{eq1-1} in order to represent the error.  Estimates for the Green's function in various norms are essential to proving sharp a posteriori estimates.  The estimates that we need have been established under relatively restrictive assumptions on $\Omega$ and the streamline direction $\va$ in \cite{FrKop2d}.  Extension to more general cases appears to be technically quite challenging, so we shall prove our results under the assumption that the Green's function behaves as in \cite{FrKop2d} without giving a complete theoretical picture of the situations for which this assumption is valid.

Finally, we are not aware of previous works on pointwise or maximum-norm a posteriori error estimation for finite element methods
for convection-dominated convection-diffusion equations.  Maximum-norm a posteriori estimates for finite difference methods for one-dimensional convection-diffusion scalar problems and systems are contained in \Blue{\cite{Kop01,TL_system_sinum09,TL_2010_IJNAM}.}
%{LC17}.
Well-known a priori maximum-norm analyses of streamline diffusion finite element methods are given in \cite{JSW87, Ni90}
{for regions away from layers.
There is also
a considerable literature on $\eps$-uniform maximum-norm error bounds for finite difference methods
on a-priori-chosen layer-adapted meshes; see, e.g.,
\cite{NK_EOR}, \cite[\S{}III.2]{RStTob} and references therein.
}
Thus there is longstanding interest in controlling pointwise errors in the presence of layer phenomena.

An outline of the paper is as follows.  In Section \ref{sec:prelims} we give preliminaries, including
the definition of the unstabilized finite element method and discussion of Green's functions.  In Section \ref{sec_conform} we prove a posteriori error estimates that are reliable and efficient for the sum of the maximum norm and an appropriate mesh-dependent seminorm of the error in the convective derivative, up to a data oscillation term.  In Section \ref{sec:stabilization} we consider a few different stabilization schemes and prove that our estimates remain valid for them under suitable assumptions.  Finally, in Section~\ref{sec:numerics} we present numerical experiments that illustrate the behavior and performance of our estimators in the context of uniform and adaptive mesh refinement.

\section{Preliminaries}

\label{sec:prelims}

In this section we discuss the analytical framework for our results and also give some finite element tools.

\subsection{Notation}
We write
 $\alpha\simeq \beta$ when $\alpha \lesssim \beta$ and $\alpha \gtrsim \beta$, and
$\alpha \lesssim \beta$ when $\alpha \le C\beta$ with a generic constant $C$ depending on $\Omega$, \Blue{$\va$, $b$, } and
%$C_f$ (and possibly $\bar C_f$),
$f$,
but not %on other essential quantities.
%
 %In particular,
 %$C$ does not depend
 on  $\eps$ or
 the diameters  of mesh elements.
  For $\mathcal{D}\subset\Omega $, $1 \le p \le \infty$, and $k \ge 0$,
  let $\|\cdot\|_{p\,;\mathcal{D}}=\|\cdot\|_{L_p(\mathcal{D})}$
  and $|\cdot|_{k,p\,;\mathcal{D}}=|\cdot|_{W_p^k(\mathcal{D})}$, where $|\cdot |_{W_p^k(\mathcal{D})}$
  is the standard Sobolev seminorm with integrability index $p$ and smoothness index $k$,
  while
    $\langle\cdot,\cdot\rangle$ is the $L_2(\Omega)$ inner product.

\subsection{Green's functions}

  {
As is standard in the literature on maximum-norm error bounds in FEM, we employ a Green's function in order to represent the error pointwise.
Let $G=G(x,\cdot)$ be the Green's function associated with \eqref{eq1-1}. For each fixed $x\in\Omega$, it satisfies
\beq\label{Green_prob}
      L^*G=-\eps\Delta G-\va\cdot\nabla G+b G=\delta(x-\cdot)\;\;\hbox{in } \Omega, \qquad
                              G=0\;\; \hbox{on } \partial \Omega,
\eeq
    where $\delta(\cdot)$ is the $n$-dimensional Dirac $\delta$-distribution.

Therefore, the unique solution $u$ of \eqref{eq1-1} allows the representation
   \beq\label{eq:sol_prim}
     u(x)=\langle G(x,\cdot),\,f(\cdot)\rangle.
   \eeq
  Similarly, any sufficiently smooth $v$
  allows the representation
\beq\label{v_via_green}
v(x)=\epsilon \langle\nabla v,\, \nabla G(x, \cdot)) + \langle\Div(\va v)+bv,\, G(x, \cdot)\rangle.
\eeq
Setting $v:=u_h$ in \eqref{v_via_green} and then subtracting \eqref{eq:sol_prim}, we immediately arrive at the error representation
\beq\label{er_via_green}
(u_h-u)(x)=\epsilon \langle\nabla u_h,\, \nabla G(x, \cdot)\rangle + \langle\Div(\va u_h)+bu_h-f ,\, G(x, \cdot)\rangle.
\eeq

Similar to \cite{DK15}, we shall rely on a number of bounds on the Green's function in which the dependence on the singular perturbation parameter $\eps$ is shown explicitly.
It is worth noting that \cite{DK15} addressed singularly perturbed equations of reaction-diffusion
type, for which the Green's function in the unbounded domain is
(almost) radially symmetric and exponentially decaying away from the singular point.
By contrast, the Green's function for the convection-diffusion
problem \eqref{eq1-1} exhibits
a much more complex %and strongly
{\it anisotropic} structure (see Fig.~\ref{fig:Green}).
%Hence, the bounds on the derivatives of $G$ are also anisotropic.
%
   \begin{figure}[tb]
      \centerline{
     \includegraphics[width=0.5\textwidth]{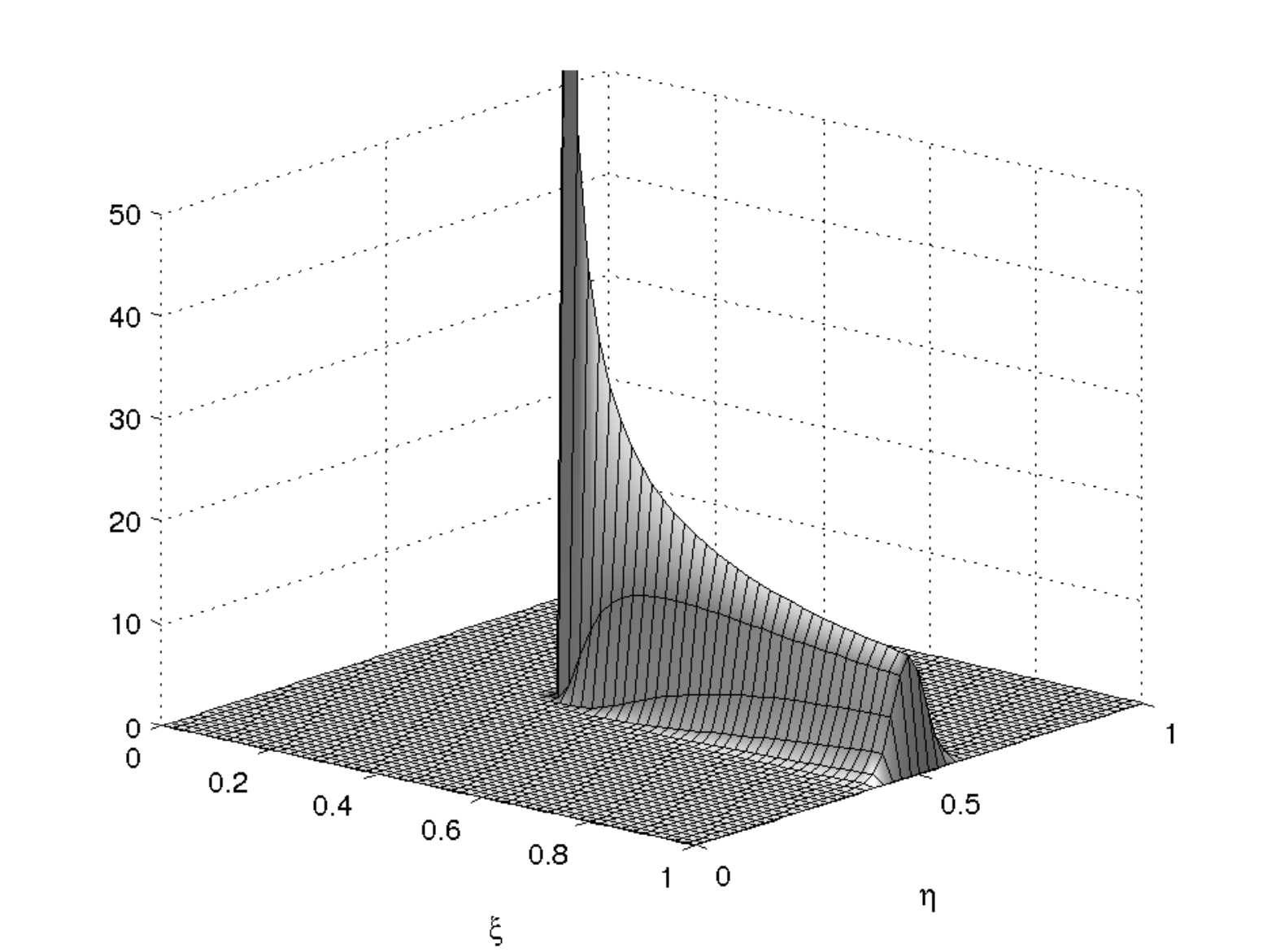}   }
      \caption{Typical anisotropic behaviour of the Green's function \Blue{$G(x,y; \xi, \eta)$} for problem \eqref{eq1-1}
      in $\Omega=(0,1)^2$ with
               $\va=[-1,0]$, $b=0$, $(x,y)=(\frac13,\frac12)$ and $\eps=10^{-3}$.}
             \label{fig:Green}
   \end{figure}

We shall require the following bounds on  the Green's function $G$ from \eqref{Green_prob}:
\begin{subequations}\label{G_bounds}
\begin{align}
\|G(x,\cdot)\|_{1\,;\Omega}+\eps^{1/2}|G(x,\cdot)|_{1,1\,;\Omega} & \lesssim 1,
\label{green_L1}\\
 %& \lesssim,\label{green_W11}\\
%\|\va\cdot \nabla G(x,\cdot)\|_{1\,;\Omega} & \lesssim 1+|\ln \eps|,
%\label{green_conv}\\
|G(x, \cdot)|_{1,1\,; B(x,\rho)\cap\Omega} &\lesssim \eps^{-1} \rho,
\label{green_locW11}\\
|G(x,\cdot)|_{2,1\,;\Omega \setminus B(x,\rho))} &\lesssim
\eps^{-1}\bigl(\ln (2+\eps/\rho)+ |\ln \eps|\bigr),
\label{green_W12}
\end{align}
\end{subequations}
as well as, occasionally,
\beq
\|\va\cdot \nabla G(x,\cdot)\|_{1\,;\Omega}  \lesssim 1+|\ln \eps|.
\label{green_conv}
\eeq
Here %a fixed
$x \in \Omega$
and $\rho>0$ are arbitrary, while $B(x,\rho)$ denotes the ball of radius $\rho$ centered at $x$.
Note that \eqref{green_conv} follows from \eqref{G_bounds}. Indeed,
$\|\va\cdot \nabla G\|_{1\,;B(x,\eps)\cap\Omega}\lesssim 1$ follows from \eqref{green_locW11},
while $\|\va\cdot \nabla G\|_{1\,;\Omega\setminus B(x,\eps)}$ is easily bounded using
\eqref{green_W12} and \eqref{green_L1} combined with the differential equation  from \eqref{Green_prob}.

\Blue{In order to gain additional insight into the scaling in the bounds given in \eqref{G_bounds} and \eqref{green_conv}, note that the fundamental solution on $\mathbb{R}^3$ with $\va=[-a_1,0,0]$ is given by $G_{\mathbb{R}^3} ({\bf x}, {\bf \xi})= \frac{1}{4 \pi \eps} \frac{\exp(\frac{1}{2} a_1(\xi_1-x_1-r)/\eps)}{r}$, where $r=\sqrt{(\xi_1-x_1)^2+(\xi_2-x_2)^2+(\xi_3-x_3)^2}$.  The scalings observed above may be directly computed from this function.  The free-space fundamental solution for convection-dominated problems in two space dimensions may also be written down explicitly and scalings computed from it, but the expression is more complex.  Note also that the bounds \eqref{G_bounds} are isotropic, while a sharper bound for the convective derivative is given in \eqref{green_conv}.  This reflects the anisotropic nature of the Green's function which can be explicitly seen in the above expression for $G_{\mathbb{R}^3}$.   }

Note that although the bounds \eqref{G_bounds} appear as an assumption for our results below, they
were rigourously proved in \cite{FrKop2d, FrKop3d,FKPP} and are also shown to be sharp in \cite{FrKop_sharp})
for a particular case of $|\va|=|a_1|$ in rectangular and cubic domains.  \Blue{We hypothesize that similar results would hold on more complicated domains such as those with reentrant corners, since similar results hold for standard elliptic and singularly perturbed reaction diffusion problems on nonconvex as well as convex domains \cite{DK15}.  However, the proof techniques used for convection-diffusion problems are different than in these other cases.  Extension to more complex domains would be technically challenging, and the form of the results is not completely clear.  We also comment on the assumption $|\va|>0$ made following \eqref{eq1-1}.  This condition is needed in the proofs of \eqref{G_bounds} given in \cite{FrKop2d, FrKop3d, FKPP}.  These proofs are substantially different than those given for scaled Green's function estimates for reaction-diffusion problems in \cite{DK15}, and it is not clear how to bridge the gap between these different techniques in order to approach the case of a convection coefficient $\va$ which sometimes vanishes.  However, the scaling obtained in \eqref{G_bounds} for convection-diffusion problems is very similar to that observed for reaction-diffusion problems, so it seems likely that \eqref{G_bounds} are also valid under  a weaker assumption that $|\va|+b>0$.
}

%\footnote{I suggest we elaborate on this in the revised version - as I plan to include some tweaks in one of the arxiv versions of the paper. Also, I was thinking of extending those bounds for constant $\va$ but much more general domains...}

\subsection{Finite element space} \label{ssec_S_Gh}

Let $\T$ be a shape-regular and conforming simplicial partition of $\Omega$,
with $\mathcal E$  denoting the set of all interior
$(n-1)$-dimensional element faces.
Let the finite element space $S_h \subset H_0^1(\Omega)$
 be the set of functions which are continuous on $\Omega$, equal to $0$ on $\partial \Omega$, and polynomials of degree at most $r$ on each $T \in \T$, where $r \ge 1$ is a fixed polynomial degree.

Similarly to \cite{DG12,DK15}, we shall employ the Scott-Zhang interpolant,
denoted $G_h$, of the Green's function $G(\cdot):=G(x, \cdot)$  from \eqref{Green_prob} (where $x\in\Omega$ remains fixed).
We let $G_h$ lie
 in the space of continuous piecewise-linear functions with respect to $\T$.   Then $G_h \in S_h$ for any $r\ge 1$, and it satisfies the   local stability and approximation property
\begin{align}
|G-G_h|_{k,1\,;T} &\lesssim h_T^{j-k} |G|_{j,1\,;\omega_T}
\label{G_h}
%\\
%|G-G_h|_{k,p,T} &\lesssim h_T|G|_{k+1,p,\omega_T}
\qquad
%\mbox{for}\;\;
\forall\, T \in \T,\quad 0 \le k \le j \le 2,
%\label{eq3-1_b}
\end{align}
%\end{subequations}
%for any $0 \le k \le j \le 2$, $ 1 \le p \le \infty$
whenever the {right}-hand side of \eqref{G_h} is defined.  Here
$h_T$ is the diameter of element $T$, while
$\omega_T$ denotes the standard  patch of elements in $\T$ touching $T$ (including $T$).

\section{A posteriori error estimation in the conforming case}\label{sec_conform}

\subsection{Finite element method and error indicators}
Introduce the standard bilinear form associated with \eqref{eq1-1}:
\begin{align}
\label{biform}
\BB(u,v):=
\epsilon \langle\nabla u,\, \nabla v\rangle + \langle\Div(\va u)+bu ,\, v\rangle.
%\int_\Omega \Bigl(\eps\, \nabla u\cdot \nabla v  + \bigl[\Div(\va u)+bu\bigr]v \Bigr).
\end{align}
The standard conforming finite element method is then given by:
%Find $u_h \in S_h$ such that
\begin{align}
\label{conf_fem}
\mbox{Find~}u_h \in S_h :\qquad
\BB(u_h, v_h)= \langle f, v_h\rangle\qquad\forall\,v_h \in S_h.
\end{align}
We emphasize that this basic finite element method is not generally practical in the singularly perturbed case $\epsilon \ll 1$
(as unstabilized finite element solutions typically exhibit non-physical oscillations unless $h_T/\eps$ is sufficiently small).
We first study this unstabilized method mainly in order to understand the structure of the error.  Stabilized schemes which are more practically relevant for singularly perturbed problems are considered below.

We shall use an a posteriori error indicator defined $\forall\,T\in\T$ by
\begin{subequations}\label{ind_def}
\begin{align}
%\begin{aligned}
&\eta_\infty(T)  : = \alpha_T \|\eps \Delta u_h - \Div(\va u_h) - bu_h +f\|_{\infty\,;T}
%\\ & ~~~~
+ \beta_T \|\llbracket \nabla u_h \rrbracket \|_{\infty\,;\pt T {\backslash\pt\Omega} },\\
%\end{aligned}
\label{alpha_beta}
&\alpha_T:=\min\{1,\,\ell_{h} \eps^{-1}h_T^2\},\qquad
%\;\;\mbox{and}\;\;
\beta_T:=\min\{\eps^{1/2},\,\ell_{h} h_T\}\qquad\mbox{in}\;\; T\in\T,
\end{align}
\end{subequations}
where the definition of the logarithmic factor
$\ell_h := 1+ \ln (2 + {\eps}\underline{h}^{-1})+|\ln \eps|$,
with $\underline{h}:=\min_{T\in\T}h_T$,
is motivated by the logarithmic factors in \eqref{G_bounds}.
Here we also use the standard notation
$\llbracket\nabla u_h \rrbracket:=\nabla u_h^+\cdot n^++\nabla u_h^-\cdot n^-$
on a face shared by two elements $T^+,\, T^-\in \T$, with their respective outward normal unit vectors $n^+$ and $n^-$.

Following the analytical techniques used in \cite{DK15} to prove similar maximum-norm a posteriori error estimates for singularly perturbed reaction-diffusion problems, we shall derive (see Lemma~\ref{lem_upper_apost}) an a posteriori upper bound of the form
\begin{align}
\label{apost_upper}
\|u-u_h\|_{\infty\,;\Omega} \lesssim \max_{T \in \T} \eta_\infty(T).
\end{align}

In the reaction-diffusion case it was also possible to prove the corresponding lower a posteriori bounds (efficiency estimates) \cite{DK15}.  However, when the convection term $\Div(\va u)$ is present in the equation, one cannot expect to prove a standard lower a posteriori bound of the form $\eta_\infty(T) \lesssim \|u-u_h\|_{\infty\,;\omega_T} + {\rm osc}$, where
%\Blue{we recall that [or should we just delete a repetition of this def, as it was defined when we introduced the %quasi-interpolant?]}
%$\omega_T$ denotes the patch
%of elements in $\T$ touching $T$ (which also includes $T$)
%around $T \in \T$, and
${\rm osc}$ is a data oscillation term.  Numerical experiments outlined below confirm that such a standard lower bound indeed does not hold.
Instead, our efficiency analysis below reveals that the error bound in \eqref{apost_upper} holds true
for the error in a stronger norm, with a certain seminorm of $\va\cdot\nabla(u-u_h)$ added in the \Blue{left}-hand side.
%\Blue{[perhaps delete "(cf \cite{tob_Ver_2015} for the energy norm case)." here
%as this was already discussed?]}
Furthermore, we shall show that the latter version of \eqref{apost_upper}, with the error measured in this new stronger norm, is efficient.

\subsection{Reliability}\label{ssec_rel}

\begin{lemma}\label{lem_upper_apost}
Under assumptions \eqref{G_bounds} on $G$, the error of the computed solution $u_h$ from \eqref{conf_fem} satisfies \eqref{apost_upper} with the a posteriori error indicators $\eta_\infty(T)$ as defined in \eqref{ind_def}.
\end{lemma}

\begin{proof}
To estimate the error at any fixed $x\in\Omega$, with slight abuse of notation, let $G(\cdot):=G(x, \cdot)$ be the Green's function from \eqref{Green_prob}.
Recall the error representation \eqref{er_via_green}
and subtract \eqref{conf_fem} with $v_h:=G_h\in S_h$.
Then, with the notation $g:=G-G_h$, one gets
\beq\label{u_h_u_conform}
(u_h-u)(x)=\epsilon \langle\nabla u_h,\, \nabla g\rangle + \langle \Div(\va u_h)+bu_h-f ,\, g\rangle.
%,
%\qquad
%F_h:=\Div(\va u_h)+bu_h-f.
\eeq
Next, a standard integration by parts in each $T\in\T$ yields
\beq\label{er_via_green_}
(u_h-u)(x)=
\langle -\eps\Delta_h u_h+\Div(\va u_h)+bu_h-f ,\, g\rangle+
{\textstyle \frac12}\sum_{T\in\T}
\int_{\pt T\backslash\pt \Omega}\!\!\epsilon \llbracket\nabla u_h \rrbracket g .
%\quad
%F_h:=\Div(\va u_h)+bu_h-f.
\eeq
%where $\color{red}\llbracket\nabla u_h \rrbracket=...$

A comparison of \eqref{er_via_green_} and the desired estimate \eqref{ind_def},\,\eqref{apost_upper}
 shows that
it suffices
to prove that
\beq\label{err_int III}
I+\II:=
\|\alpha_T^{-1}g\|_{1\,;\Omega}
+\sum_{T\in\T} %\int_{\pt T\backslash\pt \Omega}\!\!
\eps\beta_T^{-1}\|g\|_{1\,;\pt T}
\lesssim 1.
\eeq
%where $\alpha_T$ and $\beta_T$ are from \eqref{alpha_beta}.
%
When estimating $I$ and $\II$, we shall, to a degree, follow \cite[\S3.2]{DK15}.  A special treatment will be required for the elements in
 $\T_0:=\{T\in\T :\omega_T\cap B(x, c h_{T_0})\neq\emptyset\}$, where $T_0\ni x$. By shape regularity, we may choose $c>0$ sufficiently small so that the number of these elements is uniformly bounded, and thus, by shape regularity, $h_T\simeq h_{T_0}$ $\forall\,T\in\T_0$.

  For $I$,  by \eqref{alpha_beta}, note that $\alpha_T^{-1}\simeq 1+\ell_{h}^{-1} \eps h_T^{-2}$,
 while for $g=G-G_h$, in view of \eqref{G_h}, one has
 $\|g\|_{1\,;T}%=\|G-G_h\|_{1\,;T}
 \lesssim \min\bigl\{\|G\|_{1\,;\omega_T},\,h_T\|\nabla G\|_{1\,;\omega_T},\,h_T^2 |G|_{2,1\,;\omega_T}\bigr\}$.
Hence,
\beq\label{I_bound_aux}
I =\|\alpha_T^{-1}g\|_{1\,;\Omega}
\lesssim\|G\|_{1\,;\Omega}+\ell_{h}^{-1} \eps
\Bigl( |G|_{2,1\,;\Omega\backslash B(x, c h_{T_0})}
+\sum_{T\in\T_0}h_T^{-1}\|\nabla G\|_{1\,;\omega_T}
\Bigr)\lesssim 1,
\eeq
\Blue{where $\alpha_T$ is understood as an elementwise-defined piecewise-constant weight.} Here we used the bound \eqref{green_L1} for $\|G\|_{1\,;\Omega}$, and then \eqref{green_W12}
for $|G|_{2,1\,;\Omega\backslash B(x, ch_{T_0})}$.
For each $T\in\T_0$, we employed  \eqref{green_locW11} with the %minimal
ball $B(x, \bar C h_T)\supset \omega_T$,
with a sufficiently large constant $\bar C$ depending only on the shape regularity of $\T$.
%\footnote{\Blue{I've replaced $\T_h$ by $\T$ here and below - to make it consistent.}}

For $\II$,
we employ a scaled trace theorem
in the form %, which yields
$\|g\|_{1\,;\pt T}\lesssim \|\nabla g\|_{1\,;T}+h_T^{-1}\|g\|_{1\,;T}$.
Combining this with
\eqref{G_h}
yields
$\|g\|_{1\,;\pt T}\lesssim
\min\{ \|\nabla G\|_{1\,;\omega_T}\,,\,h_T| G|_{2,1\,;\omega_T}\}$.
Note also that
 $ \beta_T^{-1}\simeq \eps^{-1/2}+\ell_{h}^{-1} h^{-1}_T$,
 so
  $ \eps\beta_T^{-1}\simeq \eps^{1/2}+\eps\ell_{h}^{-1} h^{-1}_T$.
Combining these observations, one gets
$$%\begin{align}\notag
  \II\lesssim \eps^{1/2} \|\nabla G\|_{1\,;\Omega} +\eps \ell_{h}^{-1}
\Bigl(| G|_{2,1\,;\Omega\backslash B(x, c h_{T_0})}
+\sum_{T\in \T_0}  h^{-1}_T \|\nabla G\|_{1\,;\omega_T}\Bigr)\lesssim 1.
%\label{III_estimation}
$$%\end{align}
%
%\lesssim \eps^2 \sum_{T\in\T}(\eps^{-1}+\ell_h^{-1} h^{-1}_T)\bigl(\|\nabla %g\|_{1\,;T}+h_T^{-1}\|g\|_{1\,;T}\bigr)
%$$
%
Compared with the above etimation of $I$, we additionally used
$\eps^{1/2} \|\nabla G\|_{1\,;\Omega}\lesssim 1$,
in view of~\eqref{green_L1}.
This completes the proof of \eqref{err_int III}.
\end{proof}

\Blue{\begin{remark}[Nonhomogeneous Dirichlet and Neumann boundary conditions]  If $u=g$ and $u_h=g_h$ on $\partial \Omega$, then  \eqref{apost_upper} holds with an additional term $\|g-g_h\|_{L_\infty(\partial \Omega}$ added to the right hand side assuming sufficient regularity of $g$.  In particular, denoting $e=u-u_h$, let $e=e_{int} + e_{\partial}$ with $e_{\partial}=g-g_h$ on $\partial \Omega$ satisfying $\eps \langle \nabla e_{\partial}, \nabla v \rangle + \langle (\Div(\va e_{\partial}) + b, v \rangle=0$, $v \in H_0^1(\Omega)$.  Then $\|e_\partial \|_{\infty \, ; \Omega} \le \|g-g_h\|_{\infty \, ;\Omega}$ by the weak maximum principle, and one can show in a manner similar to above that $\|e_{int}\|_{\infty \, ; \Omega} \lesssim \max_{T \in \T} \eta_\infty(T)$.  The case of Neumann boundary conditions is less clear.  We are unaware of a posteriori maximum norm bounds in the literature for Neumann boundary conditions even for symmetric or non-singularly perturbed problems.  Some initial groundwork for the present singularly perturbed convection-diffusion case is contained in \cite{FKPP}, where Green's functions estimates are proved assuming homogeneous Neumann conditions along the characteristic boundaries.
\end{remark}}

\subsection{Efficiency of the volume residual}\label{ssec_eff_1}

We start with the volume residual term
in \eqref{ind_def}.
The following notation will be used. In $\Omega$, let
$e:=u-u_h$, and then define the residual %. Let $e:=u-u_h$ in $\Omega$, and
\beq\label{residual}
%\qquad\mbox{and}\qquad
R_h:=-\eps \Delta u_h+\Div(\va u_h)+bu_h-f\qquad \mbox{in~any}\; T\in\T.
\eeq
Additionally, for any $T\in \T$, let $R_{h,T}$ be the $L_2$ projection of $R_h$ over $T$ onto the space of polynomials $\mathbb{P}^{r-1}(T)$ of degree $r-1$.
Also, using the barycenric coordinates $\{\lambda_i\}_{i=1}^{n+1}$ associated with $T$,
define a standard bubble function $b_T:=\prod_{i=1}^{n+1}\lambda_i^2$.

\begin{lemma}\label{lem_eff1}
%{\color{pass}Let
%$R_T:=\eps\Delta u_h-\Div(\va u_h)-bu_h+f$ in $T\in\T$,
%and $R_{h,T}$ be its $L_2$ projection over $T$ onto the space of polynomials $\mathbb{P}^{r-1}$ of degree $r-1$.
%Then}
For any
$u_h\in S_h$ in any
$T\in\T$, one has
\beq
\label{lower_f}
\alpha_T\|R_h\|_{\infty\,;T}\lesssim
\ell_h \|e\|_{\infty\,;T}
+\alpha_T\sup_{{\psi \in \mathbb{P}^{r-1}(T):}%\atop
{ \|\psi\|_{1\,;T}=1}} \left|\int_T b_T \psi\, \va\cdot\nabla e\right|
+\alpha_T\|R_h-R_{h,T}\|_{\infty\,;T}\,.
%\quad\forall\,T\in\T.
%\vspace{-0.2cm}
\eeq
\end{lemma}

\begin{proof}
%Given %a simplex
%$T \in \T$, let
%$R_T:=\eps\Delta u_h-\Div(\va u_h)-bu_h+f$ in $T$, and $R_T\approx R_{h,T}\in \mathbb{P}^{r-1}(T)$, i.e. $R_{h,T}$
%is a polynomial of degree $r-1$ on $T$
%(e.g., let it be the $L_2$ projection of $R_T$ onto $\mathbb{P}^{r-1}(T)$).
%
%Next,
%Let $\{\lambda_i\}_{i=1}^{d+1}$ be the barycenric coordinates,
%and define a standard bubble function
For the bubble function
$b_T=\prod_{i=1}^{n+1}\lambda_i^2$,
%and set $v:=(-\eps^2\Delta u_h+f_{h,T})b_T$.
standard arguments show
that $\|R_{h,T}\|_{\infty\,;T}\|R_{h,T}\|_{1\,;T}\lesssim \|R_{h,T}\|_{2\,;T}^2\lesssim \int_T b_T R_{h,T}^2$
\cite[\S1.3.4]{Ver_book_13}.
%Combining this with $\|R_{h,T}\|_{1\,;T}\le |T|\|R_{h,T}\|_{\infty\,;T}$ implies that
%$\|R_{h,T}\|_{\infty\,;T}$
%
Next, for the function $w_T:=b_T R_{h,T}\|R_{h,T}\|^{-1}_{1\,;T}\in \mathbb{P}^{2n+r+1}(T)$, note that
both $w_T$ and $\nabla w_T$ vanish on $\pt T$, while
$h_T^2\|\Delta w_T\|_{1\,;T}\lesssim \|w_T\|_{1\,;T}\lesssim 1$.
Additionally, one gets $\|R_{h,T}\|_{\infty\,;T}\lesssim \int_T w_T R_{h,T} $, which yields
\beq\label{aux22}
\|R_h\|_{\infty\,;T}\lesssim \int_T w_T R_h
+\|R_h-R_{h,T}\|_{\infty\,;T}.
\eeq
Here $R_h=\eps\Delta e-\Div(\va e)-be$
%with $f=-\eps \Delta u+\Div(\va u ) +bu$
in terms of $e=u-u_h$, so integrating by parts twice the term with $\Delta$,
one gets
\begin{align}\notag
\int_T w_T R_h &= \int_T w_T (\eps\Delta e-\Div(\va e)-be) \\
%&= |T|^{-1}\int_T \Bigl(-\eps e\Delta \psi_T+(\Div(\va e)+be)\psi_T\Bigr)\\
&= \int_T \bigl(\eps \Delta w_T -(\Div \va+b)w_T \bigr)e-\int_Tw_T  \va\cdot\nabla e .
\label{above_double_int_n_N}
\end{align}
Note that $\|\eps \Delta w_T -(\Div \va+b)w_T \|_{1\,;T}\lesssim \eps h_T^{-2}+1\lesssim\alpha_T^{-1}\ell_{h}$.
Also, with $\psi_T:=R_{h,T}/\|R_{h,T}\|_{1\,;T}$, one has $w_T= b_T \psi_T$
in the final term of \eqref{above_double_int_n_N}, where
$\psi_T \in \mathbb{P}^{r-1}$ and $\|\psi_T\|_{1\,;T}=1$.
It remains to
combine these two observations with \eqref{above_double_int_n_N} and \eqref{aux22}, both multiplied by $\alpha_T$.
\end{proof}

\begin{remark}  Instead of arguing as in \eqref{above_double_int_n_N}, one could integrate the convection term by parts in order to obtain $|-\int_T w_T \Div(\va e) |= |\int_T e \va \cdot \nabla w_T| \lesssim \|\nabla w_T\|_{1\,;T} \|e\|_{\infty\,; T} \lesssim h_T^{-1} \|e\|_{\infty\,; T}$.  Note that $ h_T^{-1} \alpha_T = \min\{ h_T^{-1}, \ell_h h_T \eps^{-1} \}$, which is not bounded by $\ell_h$ until $h_T \eps^{-1} \le 1$.  This argument does not yield a suitable efficiency result, thus the need for the additional term in \eqref{lower_f} .
\end{remark}

It is convenient to denote
the seminorm  of $\va\cdot\nabla e$ present in %the right-hand side of
\eqref{lower_f} by
%For any   $T\in\T$, let
\beq\label{new_norm}
|\va \cdot \nabla e|_{\ast\,;T}:=
 \alpha_T\sup_{{\psi \in \mathbb{P}^{r-1}(T):}%\atop
{ \|\psi\|_{1\,;T}=1}} \left|\int_T b_T \psi\,\va\cdot\nabla e\right|
\qquad\forall\, T\in\T.
\eeq

\begin{corollary}\label{cor_eff1}
For any
$u_h\in S_h$ in any
$T\in\T$, one has
\begin{align}\label{eff_cor1}
\alpha_T\|R_h\|_{\infty\,;T}   & \lesssim
\ell_h  \|e\|_{\infty\,;T}
+|\va \cdot \nabla e|_{\ast\,;T}
+\alpha_T\|R_h-R_{h,T}\|_{\infty\,;T}\,, \\[2pt]
\label{eff_cor11}
  |\va \cdot \nabla e|_{\ast\,;T} & \lesssim \ell_h \|e\|_{\infty\,;T}+\alpha_T\|R_h\|_{\infty\,;T}\,.
\end{align}
\end{corollary}\vspace{-0.2cm}

\begin{proof}
The first assertion \eqref{eff_cor1} is equivalent to
\eqref{lower_f}.
To prove the second, note that
\eqref{above_double_int_n_N}, as well as most estimates
in the proof of Lemma~\ref{lem_eff1}
that involve $w_T=b_T \psi_T$, also hold true for a more general
$w_T:=b_T \psi$ with an arbitrary $\psi \in \mathbb{P}^{r-1}(T)$ such that $\|\psi\|_{1\,;T}=1$.
The only exceptions are \eqref{aux22} and
 the related bound $\|R_{h,T}\|_{\infty\,;T}\lesssim \int_T w_T R_{h,T} $,
 which are no longer true.  Instead we shall now employ
$|\int_T w_T R_h|\lesssim \|R_h\|_{\infty\,;T}$.
Hence, from \eqref{above_double_int_n_N} with this more general $w_T$, one gets
the following version of \eqref{lower_f}:
$$
\alpha_T%\sup_{{\psi \in \mathbb{P}^{r-1}:}%\atop
%{ \|\psi\|_{1\,;T}=1}}
\left|\int_T b_T \psi\, \va\cdot\nabla e\right|\lesssim
\ell_h \|e\|_{\infty\,;T}
+\alpha_T\left|\int_T w_T R_h\right|,
$$
and then the desired \eqref{eff_cor11}.
\end{proof}
}

\begin{remark}
One can extend the above Lemma~\ref{lem_eff1} and Corollary~\ref{cor_eff1} for a slightly simpler bubble function $b_T=\prod_{i=1}^{n+1}\lambda_i$
(instead of $b_T=\prod_{i=1}^{n+1}\lambda_i^2$). Note that in this case, \eqref{above_double_int_n_N} will additionally involve
$\int_{\pt T}\eps e\nabla w_h \cdot n$, so the proof will be slightly more involved.
\end{remark}

%\newpage
{
\subsection{Efficiency of the jump residual}\label{ssec_eff_2}
%$\color{pass}\beta_T=\min\{\eps^{1/2},\,\ell_{h} h_T\}$
Now we turn to the jump residual term
in \eqref{ind_def}, for which
following \cite{DK15}, we modify the standard edge residual efficiency proof by employing subscale mesh elements when $h_T$ does not resolve $\sqrt{\epsilon}$.

The following notation will be used.
For an interior face $E\in\mathcal E$, shared by $T^+$ and $T^-$ in $\T$,
construct two, not necessarily shape-regular, sub-simplices $T^\pm_E\subset T^\pm$ that
%coincide with/are similar to $T^\pm$ and
share the entire face $E$ and satisfy %$|T_\eps^\pm|=(d+1)^{-1} h_{\eps}|K|$.
$$
|T_E^\pm|= \hat h_{E}|E|,\quad
\hat h_{E}:=\min\bigr\{\eps^{1/2},\, |T^+||E|^{-1},\, |T^-||E|^{-1}\bigr\}
\simeq\min\bigr\{\eps^{1/2},\, h_{T^\pm}\bigr\}.
$$
To be more specific,
unless $T^\pm_E=T^\pm$,
one may impose that the vertex of $T^\pm_E$ opposite to $E$ lies on the corresponding median line in $T^\pm$.  The \Blue{simplices} $T^\pm_E$ do not necessarily satisfy either a minimum or maximum angle condition, so it is necessary to take extra care in our arguments below at a couple of points.
%{\color{red}Next, construct an edge bubble $b_K$ with support in $T_{K}^+\cup T_{K}^-$.}

\Blue{Next let $\{\lambda^\pm_i\}_{i=1}^{n+1}$ be the barycentric coordinates in $T^\pm_{E}$.  Assume that $\lambda^\pm_{n+1}|_E\equiv 0$ so that $\{\lambda^\pm_i\}_{i=1}^{n}$ are the barycentric coordinates associated with the vertices of $E$.  For the tangential gradient along $E$ we have {$|\nabla_E \lambda_i^\pm| \lesssim |{\rm diam}\,E|^{-1} \lesssim \hat h_E^{-1}$},  whereas in the direction perpendicular to $E$ the height of the triangle is $2 \hat h_E$ and so $|\partial_{E^\perp} \lambda^\pm_i| \lesssim \hat h_E ^{-1}$ with constant depending on the shape regularity of $T^\pm$.  Finally, define a standard face bubble $b_E:=\prod_{i=1}^{n}(\lambda^\pm_i)^2$ on $T^\pm_E$, and the seminorm}
\beq\label{new_norm2}
|\va\cdot\nabla e|_{\ast\,;E}:=
 \sup_{{\varphi \in \mathbb{P}^{r-1}(E):}%\atop
{ \|\varphi\|_{1\,;E}=1}}
\hat h_E^{-1}\left|\int_{T_{E}^+\cup T_{E}^-}\alpha_T\, b_E\, \varphi\, \va\cdot\nabla e\right|.
%\qquad\forall\, T\in\T.
\eeq
Here, inside the volume integral,  $\varphi \in \mathbb{P}^{r-1}(E)$ is understood as
extended to $\mathbb{R}^n$ such that it
remains constant in the direction  normal  to $E$.

%Next, let $\{\lambda_i\}_{i=1}^{n}$ be the barycentric coordinates in $E$,
%and
%$\{\lambda^\pm_i\}_{i=1}^{n}$ be the barycentric coordinates in $T^\pm_{E}$ associated with
%the vertices of $E$. For the latter,
%%$\hat h_K\simeq \min\{\eps^{1/2},\, h_{T^+},\,h_{T^-}\}$ implies
%note that
% $|\nabla \lambda_i^\pm|\lesssim \hat h_{E}^{-1}$.  In particular, for the tangential gradient along $E$ we have {$|\nabla_E \lambda_i| \lesssim |{\rm diam}\,E|^{-1} \lesssim \hat h_E^{-1}$},
% whereas in the direction perpendicular to $E$ the height of the triangle is $2 \hat h_E$ and so $|\partial_{E^\perp} \lambda_i| \lesssim \hat h_E ^{-1}$ with constant depending on the shape regularity of $T^\pm$.
% %
%Finally, define a standard face bubble $b_E:=\prod_{i=1}^{n}(\lambda^\pm_i)^2$ on $T^\pm_E$
%(for which note that $b_E=\prod_{i=1}^{n}\lambda_i^2$ on $E$), and the seminorm
%\beq\label{new_norm2}
%|\nabla e|_{\ast\,;E}:=
% \sup_{{\varphi \in \mathbb{P}^{r-1}(E):}%\atop
%{ \|\varphi\|_{1\,;E}=1}}
%\hat h_E^{-1}\left|\int_{T_{E}^+\cup T_{E}^-}\alpha_T\, b_E\, \varphi\, \va\cdot\nabla e\right|.
%%\qquad\forall\, T\in\T.
%\eeq
%Here, inside the volume integral,  $\varphi \in \mathbb{P}^{r-1}(E)$ is understood as
%extended to $\mathbb{R}^n$ such that it
%remains constant in the direction  normal  to $E$.

%Recalling that
%$\beta_T=\min\{\eps^{1/2},\,\ell_{h} h_T\}$ (by \eqref{alpha_beta}),

\begin{lemma}\label{lem_eff2}
Let $u_h\in S_h$.
For any interior face $E=\pt T^+\cap\pt T^-$, shared by $T^+$ and $T^-$ in $\T$,
with
$\beta_E:=\min\{\eps^{1/2},\,\ell_{h} h_{T^+},\,\ell_{h} h_{T^-}\}\simeq\beta_{T^\pm}$,
one has
\beq\label{lower_jump}
\beta_E\|\llbracket\nabla u_h \rrbracket\|_{\infty\,;E}
\lesssim
\ell_h\|e\|_{\infty\,;T^+\cup T^-}
+|\va\cdot\nabla  e|_{\ast\,;E}
+\|\alpha_T R_h\|_{\infty\,;T^+\cup T^-}\,.
\vspace{-0.2cm}
\eeq
\end{lemma}

\begin{proof}
Set $J_h:=\llbracket \nabla u_h \rrbracket\in \mathbb{P}^{r-1}(E)$, for which
%and
%$h_\eps\simeq \min\{\eps^{1/2},\, h_{T^+},\,h_{T^-}\}$, e.g.,
%$h_{\eps}:=\min\{\eps^{1/2},\, (d+1)|K|^{-1}|T^+|,\, (d+1)|K|^{-1}|T^-|\}$
%(with slight abuse of notation, as $h_{\eps}$ depends on $K$).
%Next, construct two, not necessarily shape-regular, simplices $T^\pm_{\eps}\subset T^\pm$, which
%%coincide with/are similar to $T^\pm$ and
%share the face $K$ and satisfy $|T_\eps^\pm|=(d+1)^{-1} h_{\eps}|K|$.
%
%
%
standard arguments show
that $\|J_h\|_{\infty\,;E}\|J_h\|_{1\,;E}\lesssim \|J_h\|_{2\,;E}^2
\lesssim \int_{E} b_{E} J_h^2$
\cite[\S1.3.4]{Ver_book_13}.
So using the function $w_E:=b_E J_h\|J_h\|^{-1}_{1\,;E}\in \mathbb{P}^{2n+r-1}(E)$,
one gets $\|J_h\|_{\infty\,;E}\lesssim \int_{E} J_h w_E$.
Note that $J_h$, as a polynomial function on $E$, can be extended to the entire $(n-1)$-dimensional plane
that contains $E$. Next, extend it to $\mathbb{R}^n$ by letting
$J_h$ remain constant in the direction  normal  to $E$.
%(equivalently, $\nabla J_h\cdot n_K=0$).
This extends $w_E$ to $T_{E}^+\cup T_{E}^-$, with
both $w_E$ and $\nabla w_E$ vanishing on $\pt T^\pm_{E}\backslash E$.
Note also  that (as elaborated in Remark~\ref{rem_w_K} below)
\beq\label{w_K}
\hat h_E^2\|\Delta w_E\|_{1\,;T^\pm_{E}}
%\lesssim
%h_\eps\|\nabla w_K\|_{1\,;T^\pm_{\eps}}
\lesssim \|w_E\|_{1\,;T^\pm_{E}}\lesssim \hat h_E,
\qquad
\hat h_E\|\llbracket\nabla w_E \rrbracket\|_{1\,;E}\lesssim  1.
\eeq
%Additionally, one gets $\|w_h\|_{\infty\,;F} \|w_h\|_{\infty\,;F_{\eps}}\lesssim |F_{\eps}|^{-1}\int_{F_\eps} w_h \psi_F$.
%
A version of \eqref{above_double_int_n_N}
(taking into account, when carrying out the integration by parts twice,
that $w_E$ and $\nabla w_E$ do not vanish on $E$) yields
\begin{align}\notag
\int_{T_{E}^+\cup T_{E}^-} w_E R_h %&= \int_T w_T (-\eps\Delta e+\Div(\va e)+be) \\
%&= |T|^{-1}\int_T \Bigl(-\eps e\Delta \psi_T+(\Div(\va e)+be)\psi_T\Bigr)\\
&= \int_{T_{E}^+\cup T_{E}^-} \bigl(\eps \Delta w_E -(\Div \va+b)w_E \bigr)e-\int_{T_{E}^+\cup T_{E}^-} w_E  \va\cdot\nabla e
\\&\qquad\qquad{}
-\int_{E} \eps\bigl(J_h\, w_E+\llbracket\nabla w_E \rrbracket\, e\bigr),
\label{new_w_K}
\end{align}
where we used $\llbracket\nabla e \rrbracket=-\llbracket\nabla u_h \rrbracket=-J_h$ on $E$.
Next, recalling that $\|J_h\|_{\infty\,;E}\lesssim \int_{E} J_h w_E$, and also \eqref{w_K}, one gets
$$
\eps\|J_h\|_{\infty\,;E}\lesssim
%\int_{K_{\eps}} \eps\llbracket\nabla w_K \rrbracket\, e+
(\eps\hat h_E^{-1}+\hat h_E)\|e\|_{\infty\,;T^+\cup T^-}
+\left|\int_{T_{E}^+\cup T_{E}^-}\! w_E\, \va\cdot\nabla e\right|
+\hat h_E\|R_h\|_{\infty\,;T^+\cup T^-}\,.
$$
Multiply this by $\eps^{-1}\beta_E$
%=\eps^{-1}\min\{\eps^{1/2},\,\ell_h h_K\}$
and note that, in view of  $\hat h_E\le \eps^{1/2}$, one has
$\eps \hat h_E^{-1}+\hat h_E%=\eps h_{\eps}^{-1}[1+h^2_{\eps}/\eps]
\le 2\eps \hat h_E^{-1}$,
so
$\eps^{-1}\beta_E(\eps \hat h_E^{-1}+\hat h_E)\lesssim \beta_E \hat h_E^{-1}\lesssim \ell_h$.
Also, for each $T=T^\pm$, with
$\alpha_T=\min\{1,\,\ell_{h} \eps^{-1}h_T^2\}$,
note that
 $\beta_E\hat h_E\lesssim \min\{\eps,\,\ell_{h} h_T^2\}=\eps\alpha_T$ yields
$\eps^{-1}\beta_E\lesssim\hat h_E^{-1}\alpha_T$.
So
$$
\beta_E\|J_h\|_{\infty\,;E}
\lesssim
\ell_h\|e\|_{\infty\,;T^+\cup T^-}
+ \hat h_E^{-1}\left|\int_{T_{E}^+\cup T_{E}^-} \alpha_T\,w_E\,  \va\cdot\nabla e\right|
+\|\alpha_T R_h\|_{\infty\,;T^+\cup T^-}\,.
$$
It remains to note that $ w_E=b_E\varphi_E$, where $\varphi_E:=J_h\|J_h\|^{-1}_{1\,;E}$
satisfies $\|\varphi_E\|_{1\,;E}= 1$.
\end{proof}

\begin{corollary}\label{cor_eff2}
Under the conditions of Lemma~\ref{lem_eff2}, one has
\begin{align}\label{eff_cor2}
  \beta_E\|\llbracket\nabla u_h \rrbracket\|_{\infty\,;E} & \lesssim
\ell_h\|e\|_{\infty\,;T^+\cup T^-}
+|\va \cdot \nabla e|_{\ast\,;E}
+\|\alpha_T R_h\|_{\infty\,;T^+\cup T^-}\,, \\
\label{eff_cor22}
  |\va \cdot \nabla e|_{\ast\,;E}
&\lesssim
\ell_h\|e\|_{\infty\,;T^+\cup T^-}
+\beta_E\|\llbracket\nabla u_h \rrbracket\|_{\infty\,;E}
+\|\alpha_T R_h\|_{\infty\,;T^+\cup T^-}\,.
\end{align}
\end{corollary}

\begin{proof}
The first assertion \eqref{eff_cor2} is equivalent to \eqref{lower_jump}.
For the second (similarly to the proof of Corollary~\ref{cor_eff1})
we note that \eqref{new_w_K}, as well as most evaluations in the proof of Lemma~\ref{lem_eff2}, holds true for a more general $w_E:=b_E\varphi$
with an arbitrary $\varphi\in \mathbb{P}^{r-1}(E)$ subject to $\|\varphi\|_{1\,;E}= 1$.
The main exception is $\|J_h\|_{\infty\,;E}\lesssim \int_{E} J_h w_E$, which is no longer true, and
instead of which we now employ
$|\int_{E} J_h w_E|\lesssim \|J_h\|_{\infty\,;E}$, which yields \eqref{eff_cor22}.
\end{proof}

\begin{remark}[\eqref{w_K} on anisotropic elements] \label{rem_w_K}
Consider a less standard \eqref{w_K} for $T_E^+$ (as $T_E^-$ is similar).
If $\hat h_{E}=|T^+||E|^{-1}$, then $T_E^+=T^+$ is shape-regular with $\hat h_{E}\simeq h_{T^+}$, so \eqref{w_K} is standard.
Otherwise, i.e. if $\hat h_{E}<|T^+||E|^{-1}$, set $\theta:=\hat h_{E}/\{|T^+||E|^{-1}\}\simeq \hat h_{E}/h_{T^+}$
and define an affine transformation from the shape-regular $T^+$ to the anisotropic $T^+_E$  by
$$
\hat x%=O+\sum_{i=1}^{n+1} [P_i-O]\lambda_i(x)+\theta[P_{n+1}-O]\lambda_{n+1}(x)
=x-(P_{0}-\hat P_0)\,\lambda_{0}(x).
$$
Here $P_{0}$ and $\hat P_0$ are the vertices in $T^+$ and $T_E^+$ opposite to $E$, while $\lambda_0$ is the
barycenric coordinate in $T^+$ associated with $P_0$.
As $|T_E^+|=\theta |T^+|$,  the Jacobian determinant of this transformation equals $\theta$.
Additionally,  in the shape-regular $T^+$ one has
$|\nabla \lambda_0|\simeq h^{-1}_{T^+}$ and $|P_{0}-\hat P_0|\lesssim  h_{T^+}$,
so all elements in the transformation matrix are $\lesssim 1$.
Now, by Cramer’s rule, all entries of the inverse transformation matrix are  $\lesssim\theta^{-1}$, so
for a generic $v$ one gets
$|\nabla_{\hat  x} v|\lesssim \theta^{-1} |\nabla_{ x} v|$, which yields \eqref{w_K} after application of a standard inverse inequality for shape-regular elements.
\end{remark}

%\dotfill
%\end{document}

\subsection{{Seminorms} of the convective derivative and overall efficiency result}\label{ssec_norms}
Combining
the definition of $\eta_\infty(T)$ in \eqref{ind_def} with
 \eqref{eff_cor1} and \eqref{eff_cor2} yields in summary that
\begin{align}
\label{eq4}
\begin{aligned}
\eta_\infty(T)  & \lesssim \ell_h  \|e\|_{\infty\,;\omega_T}
+\max_{T' \subset \omega_T}|\va \cdot \nabla e|_{\ast\,;T'}
+\max_{E\subset\pt T}|\va \cdot \nabla e|_{*\,;E}
+{\rm osc}(\alpha_T R_h\,, \omega_T),
\end{aligned}
\end{align}
where
\begin{align}
\label{osc_def}
{\rm osc}(\alpha_T R_h\,, \omega_T):= \max_{T' \subset \omega_T} \| \alpha_{T'} (R_h-R_{h,T'})\|_{\infty\,;T'}\,.
\end{align}

Recall that here we used the elementwise seminorm definitions
\eqref{new_norm} and \eqref{new_norm2}:
\begin{align*}
|\va \cdot \nabla e|_{\ast\,;T}&:=
 \alpha_T\sup_{{\psi \in \mathbb{P}^{r-1}(T):}%\atop
{ \|\psi\|_{1\,;T}=1}} \left|\int_T b_T \psi\,\va\cdot\nabla e\right|, \\
|\va \cdot \nabla e|_{\ast\,;E}&:=
 \sup_{{\varphi \in \mathbb{P}^{r-1}(E):}%\atop
{ \|\varphi\|_{1\,;E}=1}}
\hat h_E^{-1}\left|\int_{T_{E}^+\cup T_{E}^-}\!\!\alpha_T\, b_E\, \varphi\, \va\cdot\nabla e\right|.
\end{align*}
These are seminorms because it is possible that $\va \cdot \nabla e \neq 0$ is orthogonal to $\mathbb{P}^{r+2n+1} \ni b_T \psi, b_E \varphi$ over the relevant volumes.
We now define the following related global seminorm of the convective derivative:
\begin{align}
\label{star_def}
 |\va \cdot \nabla e|_{\ast} &:= \max_{T \subset \T} |\va \cdot \nabla e|_{\ast\,;T}+\max_{E\subset\mathcal E}|\va \cdot \nabla e|_{*\,;E}\,.
\end{align}

From \eqref{eq4} and \eqref{osc_def} one immediately gets the global estimate
$$
\max_{T \in \T} \eta_\infty(T)\lesssim \ell_h  \|e\|_{\infty\,;\Omega}+
 |\va \cdot \nabla e|_{\ast}
+{\rm osc}(\alpha_T R_h\,, \Omega).
$$
On the other hand, combining \eqref{apost_upper} with
 \eqref{eff_cor11} and \eqref{eff_cor22}, one gets
 $ |\va \cdot \nabla e|_{\ast}
 \lesssim \ell_h  \|e\|_{\infty\,;\Omega}+\max_{T \in \T} \eta_\infty(T)$, and then
 $$
\|e\|_{\infty\,;\Omega}+ \ell_h^{-1}|\va \cdot \nabla e|_{\ast}
% \lesssim \ell_h  \|e\|_{\infty\,;\Omega}+\max_{T \in \T} \eta_\infty(T)
 \lesssim \max_{T \in \T} \eta_\infty(T).
 $$
Combining these relationships with ${\rm osc}(\alpha_T R_h, T) \le \eta_\infty(T)$, we have proved the following for the standard conforming finite element method.
\begin{theorem}\label{theo_main}
Under assumptions \eqref{G_bounds} on $G$, the error of the computed solution $u_h$ from \eqref{conf_fem} satisfies
\begin{align}
\notag
%\begin{aligned}
\|u-u_h\|_{\infty\,;\Omega}  + \ell_h^{-1}|&\va\cdot\nabla (u-u_h)|_{\ast} + {\rm osc}(\alpha_T R_h\,, \Omega)  \lesssim  \max_{T \in \T} \eta_\infty(T)
\\ & \lesssim \ell_h\|u-u_h\|_{\infty\,;\Omega}   +|\va\cdot\nabla (u-u_h)|_{\ast}  + {\rm osc}(\alpha_T R_h\,, \Omega).
\label{eq10}
%\end{aligned}
\end{align}
\end{theorem}

Thus while our original intention was to bound $\|u-u_h\|_{\infty\,;\Omega}$, the estimator that we have naturally derived is {\it not} efficient for this norm.  Deriving an upper bound for $\eta_\infty$ instead requires inclusion of the $|\cdot|_{\ast}$ {seminorm}.  Recall also that a similar situation is observed when bounding energy norms in singularly perturbed convection-diffusion problems, where a dual norm of the convective derivative plays a similar role in the analysis \cite{tob_Ver_2015}.  Our numerical experiments below also highlight the importance of the $|\cdot|_{\ast}$ {seminorm} in a posteriori analysis for convection-diffusion problems.

We finish this section with a further discussion of the $|\cdot|_\ast$ seminorm, as its definition involves multiple terms and its meaning may not be intuitively clear at first glance.
Consider first the simpler norm %$\|\alpha_T\,\va\cdot\nabla e\|_{\infty\,;\Omega}$,
\begin{align}
\label{starstar_def}
|\va \cdot \nabla e|_{\ast\ast} & := \|\alpha_T\,\va\cdot\nabla e\|_{\infty\,;\Omega},
%\max_{T \in \T}\, \alpha_T \|\va\cdot\nabla e\|_{\infty\,;T}\,.
\end{align}
of the convective derivative, where $\alpha_T=\min\{1,\,\ell_{h} \eps^{-1}h_T^2\}\le 1$ is understood as an elementwise-defined piecewise-constant weight.  The above definitions easily yield that $|\va\cdot\nabla v|_{\ast} \lesssim |\va\cdot\nabla v|_{\ast\ast}$.  Thus our estimator measures the error in a mesh-dependent norm that lies between $\|e\|_{\infty\,;\Omega}$ and $\|e\|_{\infty\,;\Omega}+\|\alpha_T\,\va\cdot\nabla e\|_{\infty\,;\Omega}$.  The $|\cdot|_{\ast\ast}$ norm is still mesh-dependent, but gives a more transparent measure of the convective derivative of the error.  It is not generally true that $|\va\cdot\nabla v|_{\ast \ast} \lesssim |\va\cdot\nabla v|_{\ast}$.  It is straightforward to instead prove that $\alpha_T \|P_{r-1} (\va \cdot \nabla e)\|_{\infty \,; T} \lesssim |\va \cdot \nabla e|_{\ast \,;T}$ and thus that
\begin{align}
\label{norm_compare}
|\va \cdot \nabla e|_{\ast\ast} \lesssim {\max_{T \in \T} |\va \cdot \nabla e|_{\ast \,; T}}
%\Blue{|\va \cdot \nabla e|_{\ast }}
+ \max_{T \in \T} \alpha_T \|\va \cdot \nabla e-P_{r-1} (\va \cdot \nabla e)\|_{\infty \,; T},
\end{align}
where $P_{r-1}$ is the $L_2$ projection onto the elementwise polynomials of degree $r-1$.  Thus the $\ast$-{seminorm} bounds the $\ast\ast$-norm only up to an oscillation term.  In fact, $|\va\cdot\nabla u|_{\ast \ast}$ and hence $|\va \cdot \nabla e|_{\ast \ast}$ may even be unbounded in cases where $|\va\cdot\nabla u|_\ast$ and $|\va \cdot \nabla e|_\ast$ are finite.  In particular, $|\va \cdot \nabla u|_{\ast \ast}<\infty$ requires $\va \cdot \nabla u \in L_\infty$ whereas $|\va \cdot \nabla u|_{\ast} <\infty$ requires only $\va \cdot \nabla u \in L_1$.  The latter is true but not the former for example when $\Omega$ is a nonconvex polygonal domain and thus $\nabla u$ is unbounded at reentrant corners.     In the latter case a minor modification of $|\cdot|_{\ast \ast}$ by inclusion of an appropriate local bubble again results in a finite quantity.  It is also possible to define other norms between $|\cdot|_\ast$ and $|\cdot|_{\ast \ast}$.  In doing so there appears to be a tradeoff between simplicity and transparency on the one hand and fidelity to actual estimator behavior on the other, with no obvious choice doing an outstanding job at both of these tasks.

In spite of its drawbacks the $\ast \ast$ norm is easily understood as an elementwise-weighted norm of the error, tracks the $\ast$-seminorm closely in some situations, and provides clear insight into the convergence behavior of the $\ast$ seminorm.  To illustrate the latter point, assume momentarily for simplicity that the mesh is quasi-uniform with diameter $h$,  ignore logarithmic factors, and assume that {$h^2 \lesssim \epsilon$.  Then $|\va \cdot \nabla e|_{\ast \ast} \simeq  \frac{h^2}{\epsilon} \|\va\cdot \nabla e\|_{\infty\,; \Omega}$. Now,}
  as $h \rightarrow 0$ and $h\lesssim\eps$,
  we get $|\va \cdot \nabla e|_{\ast \ast} \ll h \|\va \cdot \nabla e\|_{\infty\,;\Omega}$.  Typically the maximum norm converges with one power of $h$ faster than the $W_\infty^1$ norm, so we may roughly expect $|\va \cdot \nabla e|_{\ast \ast}$ to be equivalent to $\|e\|_{\infty\,;\Omega}$ when $h \simeq \epsilon$, and for the latter to dominate the former as $h/\epsilon \rightarrow 0$.

Alternatively we may integrate by parts \eqref{new_norm} and \eqref{new_norm2} to obtain
\begin{align}
\label{eq12}
\begin{aligned}
|\va \cdot \nabla e|_{\ast} & \lesssim \max_{T\in\T}\Bigl(\alpha_T h_T^{-1}+ \alpha_T \min\{\eps^{1/2},\, h_{T}\}^{-1}\Bigr)\, \|e\|_{\infty\,;{T}}
\\ &  \lesssim \ell_h \eps^{-1}\max_{T\in\T}\Bigl(\min\{\eps^{1/2},\, h_{T}\}\Bigr)\,  \|e\|_{\infty\,;{T}} \,.
\end{aligned}
\end{align}
Here we also used $\alpha_T\le\ell_h\eps^{-1}\min\{\eps,\,h_T^2\}$.  We thus see that up to log factors, $|\va \cdot \nabla e|_{\ast} \lesssim \|e\|_{\infty\,;\Omega}$ when $h_T \le \epsilon$ (as the latter also implies
$h_T \le \epsilon^{1/2}$).   In addition, $\frac{|\va \cdot \nabla e|_{\ast}}{\|e\|_{\infty\,;\Omega}} \rightarrow 0$ as $\max_{T \in \T} h_T \rightarrow 0$.   In other words, the extended norm in \eqref{eq10} is dominated by the maximum norm over areas of $\Omega$ where the local mesh size resolves $\epsilon$.  If the problem is not singularly perturbed (i.e. $\epsilon \simeq 1$), then this heuristic is valid on any mesh.}

Both of the preceding analyses indicate that $|\va \cdot \nabla e|_\ast$ may play an important role in understanding the behavior of the maximum-norm error estimator $\max_{T \in \T} \eta_\infty(T)$
when $h_T\gg\epsilon$, but diminishes in importance relative to $\|u-u_h\|_{\infty \,; \Omega}$ and $\max_{T \in \T} \eta_\infty(T)$ as $h_T$ resolves $\epsilon$.  We verify this behavior in our numerical experiments below, and additionally provide a computational comparison between $|\cdot|_\ast$ and $|\cdot|_{\ast\ast}$.

\section{A posteriori error estimation for stabilized methods}
\label{sec:stabilization}

In this section we explore stabilization schemes and their effects on the above a posteriori error estimates.  Stabilized methods frequently have the form:
Find $u_h \in S_h$ such that
\begin{align}
\label{stab_def}
\BB(u_h, v_h)+ S_\T(u_h, v_h) = \langle f, v_h\rangle\qquad\forall\,v_h \in S_h,
\end{align}
where the stabilization term  is described using $S_\T : S_h \times S_h \rightarrow \mathbb{R}$.
%is a {\color{red}bilinear form on $S_h \times S_h$ - NO!!!}.

In order to develop a posteriori error estimates for stabilized schemes of the form \eqref{stab_def}, we
imitate the proof of
Lemma~\ref{lem_upper_apost} and use the Green's function
 $G(\cdot):=G(x, \cdot)$ and its interpolant $G_h\in S_h$.
 We again recall the error representation \eqref{er_via_green}
and subtract \eqref{stab_def} with $v_h:=G_h\in S_h$.
Then, with the notation $g:=G-G_h$, one gets
$$
(u_h-u)(x)=\epsilon \langle\nabla u_h,\, \nabla g\rangle + \langle \Div(\va u_h)+bu_h-f ,\, g\rangle
- S_\T(u_h, G_h),
%,
%\qquad
%F_h:=\Div(\va u_h)+bu_h-f.
$$
i.e. compared with \eqref{u_h_u_conform}, we have an additional term $S_\T(u_h, G_h)$.
Lemma~\ref{lem_upper_apost} for this case then reads as
\beq\label{apost_upper_stab}
\|u-u_h\|_{\infty\,;\Omega} \lesssim \max_{T \in \T} \eta_\infty(T) + \sup_{x \in \Omega} |S_\T(u_h, G_h)|.
\eeq
Bounds for the last term depend on the stabilization method; we explore some options below.
We also note that there is a large literature on stabilization methods which we do not explore here.
We shall somewhat follow \cite{tob_Ver_2015} in our presentation (where a similar analysis is carried out for the energy norm) and refer to that work for more discussion.
Hence, our exploration of this topic is
cursory and focused only on effects on maximum-norm a posteriori error estimation.   In particular, we establish that the a posteriori error estimation framework described above remains valid for a number of stabilized methods.

\subsection{Streamline diffusion method}
The streamline diffusion method is a residual-based method introduced in \cite{HB79}
{(see also \cite[\S{}III.3.2.1]{RStTob} and references therein)}.
  Here the stabilization term has the form
\beq\label{S_sdfem}
S_\T(u_h, v_h) = \sum_{T \in \T} \delta_T\!\int_T R_h\, \va \cdot \nabla v_h,
\eeq
where $R_h=-\eps\Delta u_h+\Div(\va u_h)+bu_h-f$
is the elementwise residual from \eqref{residual}.
Here $\delta_T\ge 0$ is a user-chosen parameter.
Note one standard choice \cite{BH82, JK07}
\begin{align}
\label{streamline_def}
\delta_T  =h_T a_T^{-1}\,\xi(\textstyle\frac12 \Blue{P\!e_T}),
\quad
\xi(s):=\coth(s)-s^{-1}\simeq \min\{1,s\},
\quad a_T:=\|\va\|_{\infty\,;T},
\end{align}
where we used the local  P\'{e}clet number $P\!e_T:=\eps^{-1}a_T h_T$.
Note also that the above $\delta_T$, as well as many other standard choices, satisfies the hypothesis of Corollary~\ref{cor_sdfem} below.
%\begin{align}
%P\!e_T= \frac{\|\va\|_{\infty\,;T} h_T}{2 \epsilon}.
%\end{align}
%Then we may take $\vartheta(s):=\coth(s)-s^{-1}\simeq \min\{1,s\}$ for $s\ge 0$

\begin{lemma}\label{lem_SDFEM}
Suppose $G$ satisfies \eqref{G_bounds}, and $G_h\in S_h$
is its interpolant from \eqref{G_h}.
 Then for \eqref{S_sdfem} one gets
% one gets %$|S_\T(u_h, G_h)|\lesssim \|\gamma_T\delta_T R\|_{\infty\,;\Omega}$ for any $x\in\Omega$
\beq\label{s_sdfem_bound}
|S_\T(u_h, G_h)|\lesssim \max_{T\in\T}\bigl\{\gamma_T\delta_T \|R_h\|_{\infty\,;T}\bigr\},
\;\;%\mbox{where}\;\;
\gamma_T:=
\min\{a_Th_T^{-1},\,%\eps^{-1/2},\,
\ell_h(1+a_T\eps^{-1}h_T)\}.
\eeq
\end{lemma}

\begin{corollary}\label{cor_sdfem}
Suppose that $u_h$ satisfies \eqref{stab_def},\,\eqref{S_sdfem} with
$\delta_T\lesssim h_T\min\{a_T^{-1},\,\eps^{-1}h_T\}$ $\forall\,T\in\T$.
Then, under the conditions of Lemma~\ref{lem_SDFEM},
one has
$|S_\T(u_h, G_h)|\lesssim \max_{T \in \T} \eta_\infty(T)$ for any $x \in \Omega$, and, hence,
the error bound \eqref{apost_upper}.
\end{corollary}

\begin{proof}
In view of \eqref{apost_upper_stab}, it suffices to establish the desired bound on $S_\T(u_h, G_h)$.
For the latter, a comparison of \eqref{s_sdfem_bound} with \eqref{ind_def} shows that it suffices to prove that $\gamma_T\delta_T\lesssim \alpha_T=\min\{1,\,\ell_{h} \eps^{-1}h_T^2\}$.
From $\gamma_T\lesssim a_Th_T^{-1}$ combined with
$\delta_T\lesssim h_Ta_T^{-1}$ one immediately gets $\gamma_T\delta_T\lesssim 1$, so it remains to prove that we also have
$\gamma_T\delta_T\lesssim\ell_{h} \eps^{-1}h_T^2$.
The latter follows by combining $\gamma_T\lesssim \ell_h(1+a_T\eps^{-1}h_T)\simeq \ell_h\max\{1,\,a_T\eps^{-1}h_T\}$
with $\delta_T\lesssim h_Ta_T^{-1} \min\{1,\,a_T\eps^{-1}h_T\}$
(in view of $\min\{1,s\}\max\{1,s\}=s$ $\forall\,s$).
\end{proof}

{\it Proof of Lemma~\ref{lem_SDFEM}.}
A comparison of the desired bound \eqref{s_sdfem_bound} with \eqref{S_sdfem} shows that
it suffices to prove that
$$
I^*:=\sum_{T\in\T}I^*_T:=\sum_{T\in\T}\gamma_T^{-1}\|\va \cdot \nabla G_h\|_{1\,;T}\lesssim 1.
$$
Note that here $\gamma_T^{-1}\simeq
h_T a_T^{-1}+
\ell_h^{-1}\min\{1,\eps h_T^{-1}a_T^{-1}\}$.
Hence, a calculation using $G_h=G-(G-G_h)$ leads to
 $$
I^*_T\lesssim
 h_T\| a_T^{-1}\va \cdot \nabla G_h\|_{1\,;T}
 +\ell_h^{-1}\|\va \cdot \nabla G\|_{1\,;T}
 +\ell_h^{-1}\eps h_T^{-1}\| a_T^{-1}\va \cdot \nabla( G-G_h)\|_{1\,;T}\,.
 $$
 Here $|a_T^{-1}\va|\le 1$. Additionally, for the first term, an inverse inequality applied elementwise yields
 $h_T\|  \nabla G_h\|_{1\,;T}\lesssim \|G_h\|_{1\,;T}\lesssim \|G\|_{1\,;\omega_T}$, where we also used \eqref{G_h}.
 For the final term, \eqref{G_h} implies
 $\| \nabla ( G-G_h)\|_{1\,;T}\lesssim \min\{\|\nabla G\|_{1\,;\omega_T},\,h_T\|D^2 G\|_{1\,;\omega_T}\}$.
Combining these observations, one now gets
 $$
 I^*\lesssim
 \| G\|_{1\,;\Omega}
 +\ell_h^{-1}\|\va \cdot \nabla G\|_{1\,;\Omega}
+ \ell_{h}^{-1}\eps
\Bigl(| G|_{2,1\,;\Omega\backslash B(x, c  h_{T_0})}
+\sum_{T\in \T_0}  h^{-1}_T \|\nabla G\|_{1\,;\omega_T}\Bigr),
 $$
where
 we again used $T_0\ni x$ and
 $\T_0=\{T\in\T :\omega_T\cap B(x, \Blue{c} h_{T_0})\neq\emptyset\}$.
 Most ingredients of the right-hand side have been estimated in \eqref{I_bound_aux}.
The remaining $\ell_h^{-1}\|\va \cdot \nabla G\|_{1\,;\Omega}$ is estimated using \eqref{green_conv},
which yields the desired bound $I^*\lesssim 1$. (Note that
 \eqref{green_conv} follows from \eqref{G_bounds}.)
\endproof

\subsection{Continuous interior penalty stabilization}

We next let  $u_h$ satisfy \eqref{stab_def} with the stabilizing term
{\cite{DD76, Burman_Hansbo_2004} (see also \cite[\S{}III.3.3.2]{RStTob}, \cite[\S2.2.4]{tob_Ver_2015}
and references therein)}
\beq\label{int_penalty}
S_\T (u_h, v_h) = \sum_{E \in \E} \tau_E\! \int_E \llbracket \va \cdot \nabla u_h \rrbracket \llbracket \va \cdot\nabla v_h \rrbracket.
\eeq
Here we used the standard notation $\llbracket \cdot \rrbracket$, which, for a generic scalar function $v$,
is defined by
$\llbracket v \rrbracket:=\llbracket v n_E \rrbracket$ on any $E\in\mathcal E$ using any fixed normal unit vector $n_E$
to $E$.
A user-chosen parameter $\tau_E$ typically satisfies
\beq\label{int_penalty_tau}
\tau_E \lesssim h_E^2.
\eeq

Following the analysis in \cite[Lemma~2.6]{tob_Ver_2015}, we restrict our consideration to the case of $\mathbb{P}^1$ elements
and, thus, get the following result.

\begin{lemma}\label{lem_int_pen}
Suppose that
$u_h$ satisfies \eqref{stab_def},\,\eqref{int_penalty},\, \eqref{int_penalty_tau}
with the space $S_h$ of $\mathbb{P}^1$ elements,
and that the functions $\va$, \Blue{$\Div \va$,} $b$, and $f$ in \eqref{eq1-1}  are continuous in $\Omega$.
Suppose also that  $G$ satisfies \eqref{G_bounds}, and $G_h\in S_h$
is its interpolant from \eqref{G_h}.
 Then for \eqref{int_penalty}, one gets
$|S_\T(u_h, G_h)|\lesssim \max_{T \in \T} \eta_\infty(T)$ for any $x \in \Omega$, and, hence,
the error bound \eqref{apost_upper}.
\end{lemma}
\begin{proof}
In view of \eqref{apost_upper_stab}, it suffices to establish the desired bound on $S_\T(u_h, G_h)$.
As we consider the case  of $\mathbb{P}^1$ elements,
 $\epsilon \Delta u_h=0$ elementwise for any $u_h \in S_h$.
 Hence, the residual
becomes
$R_h=\Div(\va u_h)+bu_h-f$, while, in view of the continuity of $\va$, \Blue{$\Div \va$,} $b$, and $f$, one then gets
$\llbracket \va \cdot \nabla u_h \rrbracket = \llbracket R_h \rrbracket$.
Hence, \eqref{int_penalty} leads to
$$
|S_\T (u_h, G_h)|
%\lesssim \sum_{E\in\mathcal E} h_E^2 \int_{E}|\llbracket R_h \rrbracket| |\llbracket \nabla G_h \rrbracket|
\lesssim I^{**}\max_{T\in\T}\Bigl\{\alpha_T\| R_h\|_{\infty\,;T}\Bigr\},\quad
I^{**}:=\sum_{T\in\T}
\alpha_T^{-1}h_T^2\|\llbracket  \va \cdot  \nabla G_h \rrbracket\|_{1\,;\pt T}\,,
$$
where we also used that $h_E\simeq h_T$ for any shape-regular $T$ sharing a face $E$.
As $\alpha_T\| R_h\|_{\infty\,;T}\le \eta_\infty(T)$ (in view of \eqref{ind_def}),
it remains to show that $I^{**}\lesssim 1$.

For the latter, first, note that an inverse inequality
yields $h_T^2\|\llbracket \nabla G_h \rrbracket\|_{1\,;\pt T}\lesssim
h_T^{j} | G_h|_{j,1\,;\omega_T}$ for $j=0,1$,
while
\eqref{G_h} implies
 $ | G_h|_{j,1\,;T}\lesssim  | G|_{j,1\,;\omega_T}$.
On the other hand, $\llbracket \nabla G_h \rrbracket=-\llbracket \nabla g \rrbracket$, where $g=G-G_h$,
so a standard scaled trace inequality yields
$\|\llbracket \nabla g \rrbracket\|_{1\,;\pt T}\lesssim \|D^2 g\|_{1\,;\omega_T}+h_T^{-1}\|\nabla g\|_{1\,;\omega_T}$,
for which \eqref{G_h} gives $\|D^2 g\|_{1\,;T}+h_T^{-1}\|\nabla g\|_{1\,;T}\lesssim \|G\|_{2,1\,;\omega_T}$.
Combining these observations, one gets
\beq\label{aux_int_pen}
h_T^2\|\llbracket {\va\cdot}\nabla G_h \rrbracket\|_{1\,;\pt T}\lesssim
\min\Bigl\{\| G\|_{1\,;\omega_T'},\,h_T| G|_{1,1\,;\omega_T'},\,\,h_T^2|G|_{2,1\,;\omega_T'}\Bigr\},
\eeq
where $\omega_T'$ denotes
the patch of elements in $\T$ touching $\omega_T$ (including those in $\omega_T$).
Finally, combining the definition of $I^{**}$ with \eqref{aux_int_pen} and
$\alpha_T^{-1}\simeq 1+\ell_{h}^{-1} \eps h_T^{-2}$, one gets
$$
I^{**}\lesssim \|G\|_{1\,;\Omega}+\ell_{h}^{-1} \eps
\Bigl( |G|_{2,1\,;\Omega\backslash B(x, ch_{T_0})}
+\sum_{T\in\T'_0}h_T^{-1}\|\nabla G\|_{1\,;\omega'_T}
\Bigr),
$$
where
 $\T'_0:=\{T\in\T :\omega'_T\cap B(x, ch_{T_0})\neq\emptyset\}$, with $T_0\ni x$.
Now, the desired bound $I^{**}\lesssim 1$ is obtained similarly to \eqref{I_bound_aux}.
\end{proof}

\subsection{Local projection stabilization}
In this section, we shall discuss local projection stabilization methods (and, very briefly, somewhat related subgrid-scale schemes).
We shall see that for such methods one can choose the Green's function's interpolant $G_h\in S_h$ such that
$S_\T(u_h, G_h)=0$ in
\eqref{apost_upper_stab}, which immediately yields the a posteriori error  \eqref{apost_upper} and, hence,
a more general \eqref{eq10}.

We shall mainly focus on
 local projection stabilization methods of the form \eqref{stab_def}{---see, e.g., \cite[\S{}III.3.3.1]{RStTob}, \cite[\S2.2.2]{tob_Ver_2015} for
 further details and references therein---}%
 with the stabilizing term
\beq\label{S_loc_pr}
S_\T(u_h, v_h) = \sum_{M \in \mathcal M} \delta_M\!\int_M\kappa_h( \bar{\va}_{M} \cdot \nabla u_h)\,\kappa_h( \bar{\va}_{M} \cdot \nabla v_h).
\eeq
This stabilizing term uses fluctuations of the convective derivatives computed using the fluctuation operator $\kappa_h:=I-\pi_h$,
where $\pi_h$ is a projection onto an appropriate discontinuous finite element space related to an auxiliary partition
$\mathcal M$
of $\Omega$.
The approximation $\bar{\va}_{M}$ of $\va$ is assumed constant in each $M\in\mathcal M$.
 A user-chosen parameter $\delta_M$ in \eqref{S_loc_pr} typically satisfies
$$
\delta_M\simeq h_M\,\|\va\|^{-1}_{\infty\,;M}\,.
$$

Both ${\mathcal M}=\T$ (one-level approach) and $\T$ generated by a single-level refinement of each element in $\mathcal M$ (two-level approach) have been introduced in the literature
{\cite{BB01,BB04,Matthies_etal_2007}}.
This can be implemented in various ways. To be more precise,
 it will be convenient to denote by
 $S_h^r $  the set of functions which are continuous on $\Omega$, equal to $0$ on $\partial \Omega$, and polynomials of degree at most $r$ on each $T \in \T$, where $r \ge 1$ is a fixed polynomial degree.
 An analogous set of functions relative to the partition $\mathcal M$ will be denoted $S_H^r $.
 With this notation, we assume that the finite element space $S_h$ and the fluctuation operator $\kappa_h$ satisfy
 \beq\label{S_h_1r}
 S_h^1\subseteq S_h\subseteq S_h^r\quad\mbox{for some~~}r\ge 1,
 \qquad
 \kappa_h( \nabla v_H)=0\quad\forall\, v_H\in S_H^1\,.
 \eeq
The above assumption on $\kappa_h$ is satisfied if $\pi_h v$ in each $M\in\mathcal M$ is defined as the $L_2$ projection of $v$ onto $\mathbb{P}^{q}(M)$
for some $q\ge 0$.
In the two-level case,  one may set $S_h:=S_h^r$ and $q:=r-1$.
In the one-level case ${\mathcal M}=\T$, for some $q\ge 0$,
the space $S^{q+1}_h$ %of continuous $\mathbb{P}^{q+1}$ elements
can be enriched by appropriate bubble basis functions, so one gets \eqref{S_h_1r} with
{$r:=q+n+1$.}
%(to ensure that $\kappa_h(\nabla v_h)=0$ $\forall\, v_h\in S_h$)
%the finite element space $S_h$ by a bubble space /: $P^{r-1}$ multiplied by bubble functions...
For example, if $q=0$, one may employ $S_h=S_h^1\bigoplus {\rm span}\{b_T: T\in \T\}$,
where $b_T={\prod}_{i=1}^{n+1}\lambda_i {\in \mathbb{P}^{n+1}(T)}$ is the bubble function associated with $T$, {so \eqref{S_h_1r} is satisfied with $r=n+1$}.

For both one-level and two-level approaches under the general assumption \eqref{S_h_1r}, we get the following result.

\begin{lemma}\label{loc_pr}
Suppose that
$u_h$ satisfies \eqref{stab_def},\,\eqref{S_loc_pr} under assumption \eqref{S_h_1r}.
Then, under the conditions \eqref{G_bounds} on $G$,
 the error of the computed solution $u_h$ satisfies
\eqref{eq10},
i.e. Theorem~\ref{theo_main} remains valid for this method.
\end{lemma}

\begin{proof}
In view of \eqref{S_h_1r}, one has
 $u_h\in S_h^r$,
 so the efficiency results of \S\S\ref{ssec_eff_1}--\ref{ssec_eff_2}
 apply immediately. Hence, to establish \eqref{eq10}, it suffices to prove \eqref{apost_upper}.

 Next, construct the interpolant $G_h\in S_H^1\subseteq S_h^1$ of $G$ exactly as described in \S\ref{ssec_S_Gh}, only relative to the partition $\mathcal M$. Then the bounds \eqref{G_h} hold true, only
 with
$\omega_T$ now denoting the patch
of elements in $\mathcal M$ touching $M\supseteq T$ (which also includes this $M$).
With this tweak in the notation, the estimates in \S\ref{ssec_rel} remain valid, so
Lemma~\ref{lem_upper_apost} for this case again reads as
\eqref{apost_upper_stab}.
Finally, $G_h\in S_H^1$ combined with \eqref{S_h_1r} yields $S_\T(u_h, G_h)=0$, which, combined with \eqref{apost_upper_stab},
gives the desired \eqref{apost_upper}.
\end{proof}

Note that Lemma~\ref{loc_pr} also remains valid for a version of \eqref{S_loc_pr} with the fluctuations of the full gradient,
i.e. with $\kappa_h( \bar{\va}_{M} \cdot \nabla \cdots)$ replaced by $\kappa_h(\nabla \cdots)$ (as in this case we again enjoy $S_\T(u_h, G_h)=0$ for $G_h\in S_H^1$).

In addition, the above argument may be applied to subgrid-scale methods, in which gradients of fluctuations are used instead of
 fluctuations of gradients as in \eqref{S_loc_pr} for local projection methods \cite{Gu99,Gu01};
 {see also, e.g., \cite[\S5]{Matthies_etal_2007}, \cite[\S{}IV.4.5]{RStTob}, \cite[\S2.2.3]{tob_Ver_2015}.} For example, one may replace
the terms of type   $\kappa_h( \bar{\va}_{M} \cdot \nabla \cdots)$
in \eqref{S_loc_pr} by
 $ {\va}_{M} \cdot \nabla(\tilde\kappa_h \cdots)$ (or the full-gradient version  $ \nabla(\tilde\kappa_h \cdots)$). A typical fluctuation operator $\tilde\kappa_h$ satisfies $\ker \tilde\kappa_h\supseteq S_H^1$ (so $\tilde\kappa_h G_h=0$), in which case we again get Lemma~\ref{loc_pr}.

\subsection{Concluding remarks on stabilized methods}

Above we have established that the introduction of a variety of stabilization techniques does not affect the ability to bound $\|u-u_h\|_{L_\infty(\Omega)}$ using our residual estimator, although the form of the proof depends on the particular stabilization technique.  The rest of our arguments concerning the seminorm $|\va \cdot \nabla e|_{\ast}$ and efficiency of \Blue{our}  estimators are generally not affected by the introduction of stabilization, since they do not use Galerkin orthogonality in their proof.  The only exception comes in the choice of $r$ used to define oscillation and $|\cdot|_\ast$, which as noted in the preceding subsection may require a little bit of care when employing projection methods with bubble functions in the definition of $S_h$.  We summarize these findings below.

\begin{theorem}
Consider streamline diffusion stabilization under the assumptions of Corollary \ref{cor_sdfem}, continuous interior penalty stabilization under the assumptions of Lemma \ref{lem_int_pen}, or local projection stabilization under the assumptions of Lemma \ref{loc_pr}.  Under these conditions and the assumptions of Theorem \ref{theo_main}, the conclusions of Theorem \ref{theo_main} remain valid.  That is, the estimator $\max_{T \in \T} \eta_\infty(T)$ is reliable and efficient for the error notion $\|u-u_h\|_{\infty \,; \Omega} + |\va\cdot\nabla(u-u_h)|_\ast + {\rm osc}(\alpha_T R_h, \Omega)$ up to factors of $\ell_h$.
\end{theorem}

\section{Numerical experiments}
\label{sec:numerics}

In this section we present numerical experiments which illustrate the practical behavior of the error estimators and indicators defined above.  All computations were carried out in MATLAB using suitably modified routines from  the adaptive finite element library iFEM \cite{Ch09PP}.

In all cases we took $\Omega$ to be the unit square $(0,1) \times (0,1)$ and
{the coefficients ${\bf a}=[0,1]$ and} $b=1$.
%We also took either ${\bf a}=[0,1]$ or ${\bf a}=[1,0]$, as specified below.
We computed using affine Lagrange elements and either uniform or adaptive refinement.  In the case of adaptive refinement, we used a modified maximum strategy in the standard $\mathtt{solve} \rightarrow \mathtt{estimate} \rightarrow \mathtt{mark} \rightarrow  \mathtt{refine}$ loop.  Let $\eta_\infty=\max_{T \in \T} \eta_\infty(T)$ be the overall error estimator, and fix a (small) positive integer $K_{max}$.  An element $T \in \T$ is bisected $K_{max}$ times if $\eta_\infty(T) \ge 0.5 \eta_\infty$, $K_{max}-1$ times if $0.5 \eta_\infty > \eta_\infty(T) \ge 0.25 \eta_\infty$, etc.  $K_{\max}$ can be varied based on the degree of singular perturbation, with $K_{\max}=4$ being used in the experiments below.  This scheme helped to prevent too few elements being refined at each iteration of the adaptive procedure, and thus too many iterations from occurring.  It also aided in more efficient resolution of boundary and interior layers since elements with large indicators are subdivided multiple times in each adaptive step.

Unknown constants appear in our error estimators and must be fixed.  In our experiments we chose the definition
$$\begin{aligned}
\eta_\infty(T)& =\min  \left [1,\, 0.0125 \ell_h \frac{h_T^2}{\epsilon} \right ] \|R_T \|_{L_\infty(T)}
\\ & ~~~~+ \min \left [ \sqrt{\epsilon},\, 0.03 \ell_h  h_T \right ] \|\llbracket \nabla u_h \rrbracket \|_{L_\infty(\partial T)}.
\end{aligned}$$
Experiments were conducted using either an unstabilized scheme or streamline diffusion stabilization with the parameter chosen as in \eqref{streamline_def}.

\subsection{Experiment 1: Smooth solution}
In this experiment we consider a simple smooth solution
$$u_1(x,y)=\sin (\pi x) \sin (\pi y).
%\color{yellow}~~{\bf a}= [1,0].
$$
No stabilization was used.  Our goal here is to illustrate and compare the convergence orders of the a posteriori error estimator $\max_{T \in \T} \eta_\infty(T)$, the target error norm $\|u-u_h\|_{L_\infty(\Omega)}$, and the convective error $|\va\cdot\nabla(u-u_h)|_{*}$.  We thus take a uniform series of mesh refinements, and carry out convergence studies with $\epsilon=3 \times 10^{-3}$ and $\epsilon = 10^{-5}$.   Results are displayed in Figure \ref{fig2}.   In both cases we observe that $\|u-u_h\|_{L_\infty(\Omega)}$ converges with order $DOF^{-1}=O(h^2)$, and with the same order of magnitude observed in each case.  When $\epsilon=3 \times 10^{-3}$, we observe a preasymptotic regime in which the error estimator and convective error $|\va\cdot\nabla(u-u_h)|_{*}$ both converge with order $DOF^{-3/2}=h^3$.  As $h$ sufficiently resolves $\epsilon$, the estimator instead tracks the error $\|u-u_h\|_{L_\infty(\Omega)}$ with order $DOF^{-1}$, while the convective error measured in the $\ast$-seminorm continues to decrease with order $DOF^{-3/2}$.  In all regimes, this illustrates that $\max_{T \in \T}\eta_\infty(T) \simeq \|u-u_h\|_{L_\infty(\Omega)} + |\va\cdot\nabla(u-u_h)|_\ast$ (up to data oscillation, which here is bounded by the latter two terms).  In addition it highlights two different dominant convergence rates for the estimator.  A third initial regime of $O(DOF^{-1/2})=O(h)$ convergence for the estimator and $|\va\cdot\nabla (u-u_h)|_{\ast}$ is illustrated in the right plot in Figure \ref{fig2}, where $\epsilon=10^{-5}$.

\setlength{\unitlength}{.8cm}
\begin{figure}[t]
\begin{center}
\includegraphics[width=2.8in]{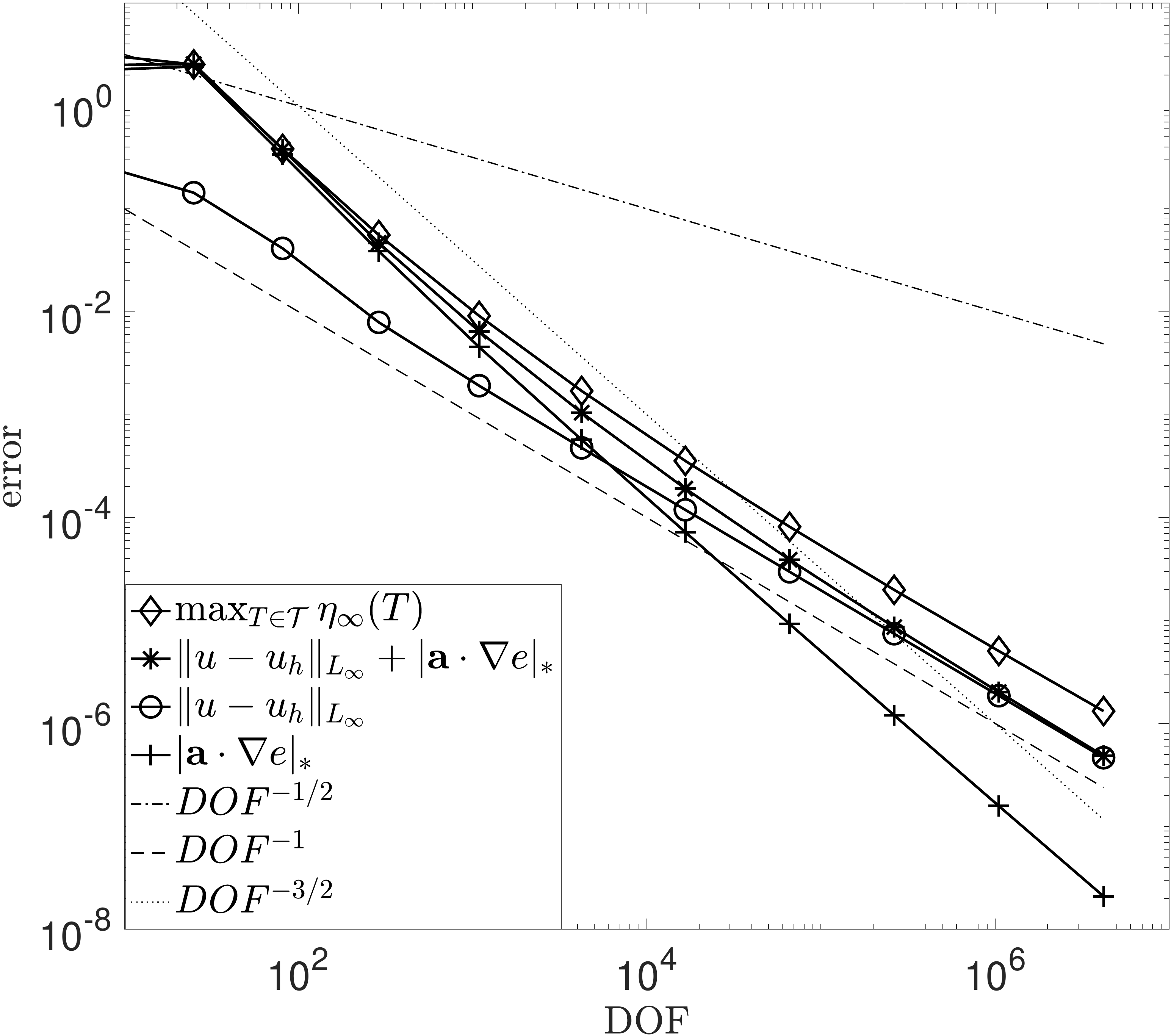}
\includegraphics[width=2.8in]{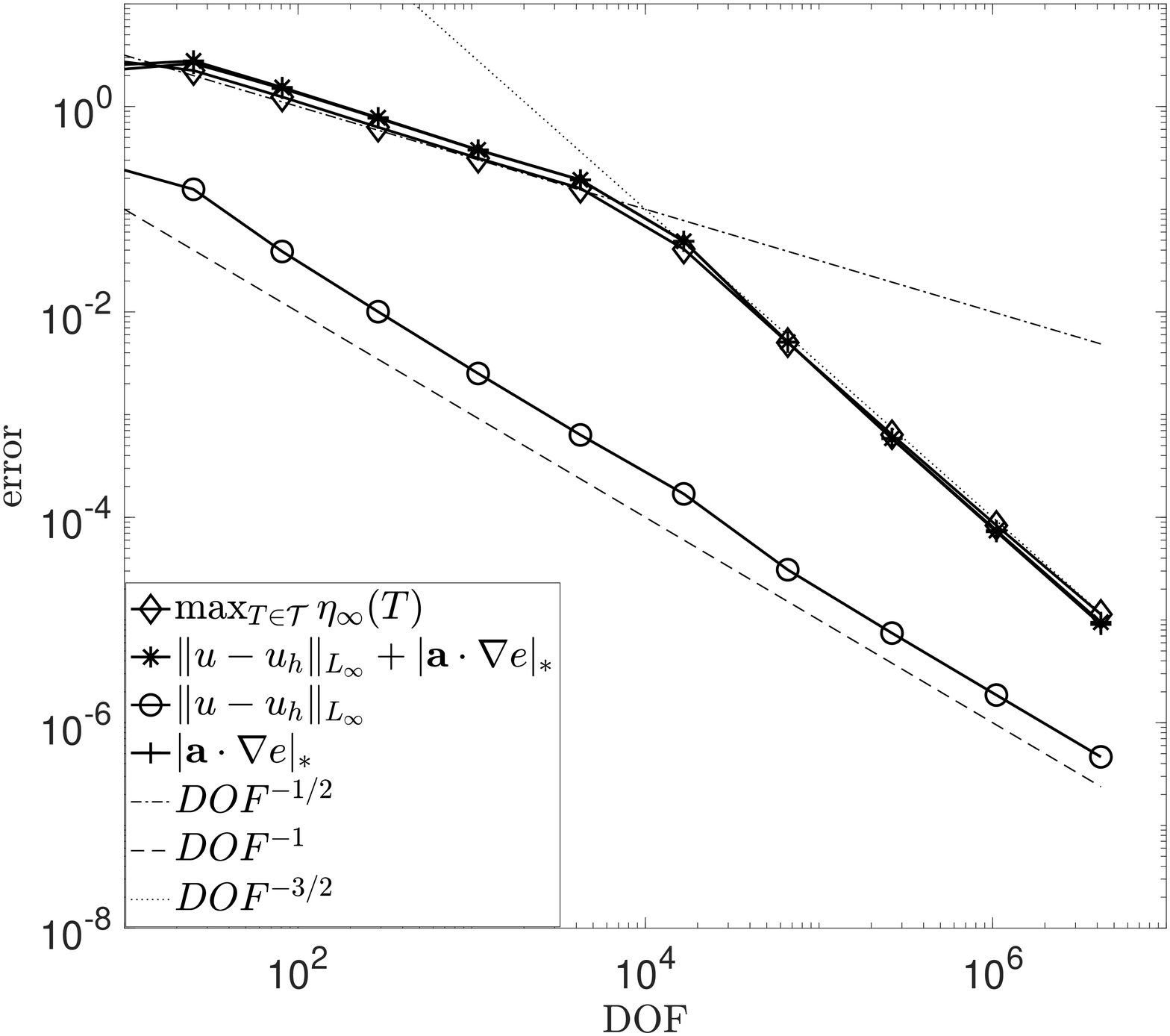}
\caption{{Smooth solution $u_1$} and uniform refinement,
{standard Galerkin method without stabilization,}
  $\epsilon=3\times 10^{-3}$ (left) and $\epsilon=10^{-5}$ (right).}
\label{fig2}
\end{center}
\end{figure}

Understanding these convergence regimes is easiest when considering the $**$-seminorm $|\va\cdot\nabla (u-u_h)|_{\ast \ast} = \sup_{T \in \T} \min  \left (1, 0.0125 \ell_h \frac{h_T^2}{\epsilon} \right ) \|\va \cdot \nabla e\|_{\infty \,; T}$, which closely tracks the $ \ast $ seminorm in this case.  We may expect $\|\va \cdot \nabla (u-u_h)\|_{L_\infty(T)}=O(h)=O(DOF^{-1/2})$.  Initially $\min(1, \ell_h \frac{h_T^2}{\epsilon})=1$, so  $|\va\cdot\nabla (u-u_h)|_{\ast \ast}=O(h)$ also.  When $h_T^2 \ell_h < \sqrt{\epsilon}$, then we have $|\va\cdot\nabla (u-u_h)|_{\ast \ast} \lesssim \frac{h^3 \ell_h}{\epsilon}$, leading to the increased rate of convergence observed in Figure \ref{fig2}.

\begin{remark}
In Figure \ref{fig2} we observe that $\|u-u_h\|_{L_\infty(\Omega)}$ converges with optimal rate $O(DOF^{-1})$ from essentially the first mesh refinement.  This implies that oscillations often associated with unstabilized solution of singularly perturbed problems are not present here.  This was confirmed by viewing plots of the discrete solution.   Lack of instability in the discrete solution may be due to symmetries in the test solution, which in 1D examples has been observed to lead to similar unexpectedly good results for unstabilized methods \cite[Chap.\,4]{Kop93}.  Although numerical stability is uncharacteristically good for this example, it
is nonetheless useful as it allows for clear exposition of the properties of our estimator relative to the target error notion $\|u-u_h\|_{\infty\,; \Omega}$ and seminorm $|\va\cdot\nabla(u-u_h)|_{*}$.
\end{remark}

\subsection{Experiment 2:  {Outflow boundary}  layer}
In this subsection we consider the {outflow boundary} layer solution
%\Blue{[below $x$ and $y$ are swapped, so that $\va$ remains the same for all examples]}
$$u_2(x,y)=x(1-x)\left [ y-\frac{e^{-(1-y)/{\epsilon}}-e^{-1/{\epsilon}}}{1-e^{-1/{\epsilon}}}                \right ].
%\color{yellow} ~{\bf a}=[0,1],
$$
This solution exhibits a strong layer at $y=1$ of width $O(\epsilon)$ and corresponding maximum {solution} gradient size $O(\epsilon^{-1})$.  In this experiment and in all of the displayed adaptive experiments involving layers we employed streamline diffusion stabilization.

In the left plot in Figure \ref{fig3}  we illustrate the advantage of adaptive versus uniform refinement with $\epsilon=10^{-5}$.  Uniform refinement leads to essentially no decrease in the target norm $\|u-u_h\|_{L_\infty(\Omega)}$, but some decrease in the $*$-seminorm and estimator.  Adaptive refinement yields little initial decrease in the maximum error, but optimal $O(DOF^{-1})$ decrease begins with a little over $10^5$ degrees of freedom.  The $*$-seminorm and estimator decrease with order $DOF^{-1}$ for most of the convergence history.

In the right plot of Figure \ref{fig3} we include data oscillation in the plot in order to illustrate that the $|\cdot|_{\ast}$-seminorm is an essential part of the error notion measured by $\max_{T \in \T} \eta_{\infty(T)}$.  In particular, we see that $\max_{T \in \T} \eta_\infty(T) \simeq |\va \cdot \nabla e |_{\ast}  \gg \|u-u_h\|_{\infty \,; \Omega}+\max_{T \in \T_{h}} \alpha_T \|R_h-R_{h,T}\|_{\infty\,; T}$ for much of the convergence history, confirming that the estimator is reliable but not efficient when measuring only the sum of the maximum error and data oscillation.

\setlength{\unitlength}{.8cm}
\begin{figure}[t]
\begin{center}
\includegraphics[width=2.8in]{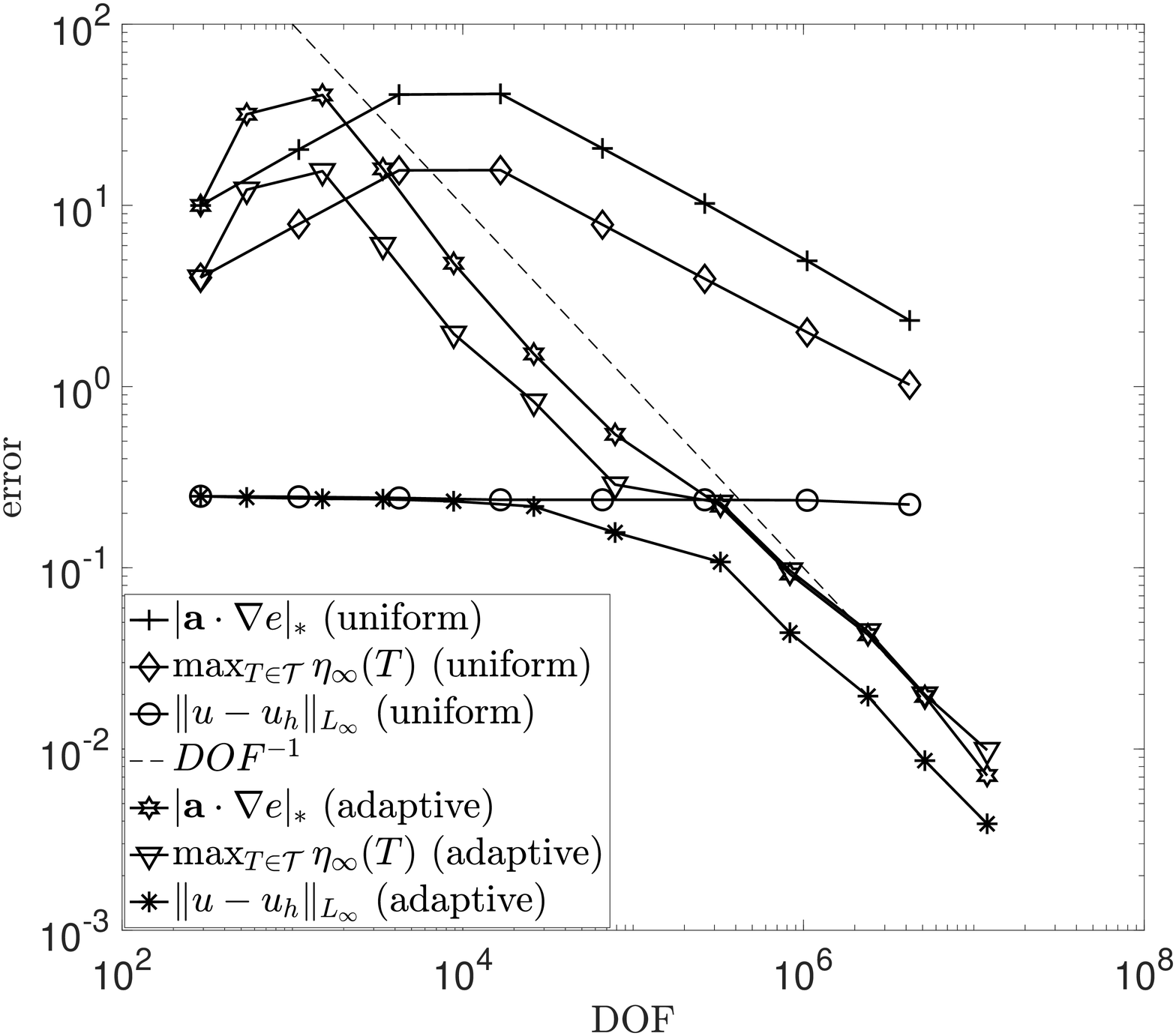}
\includegraphics[width=2.8in]{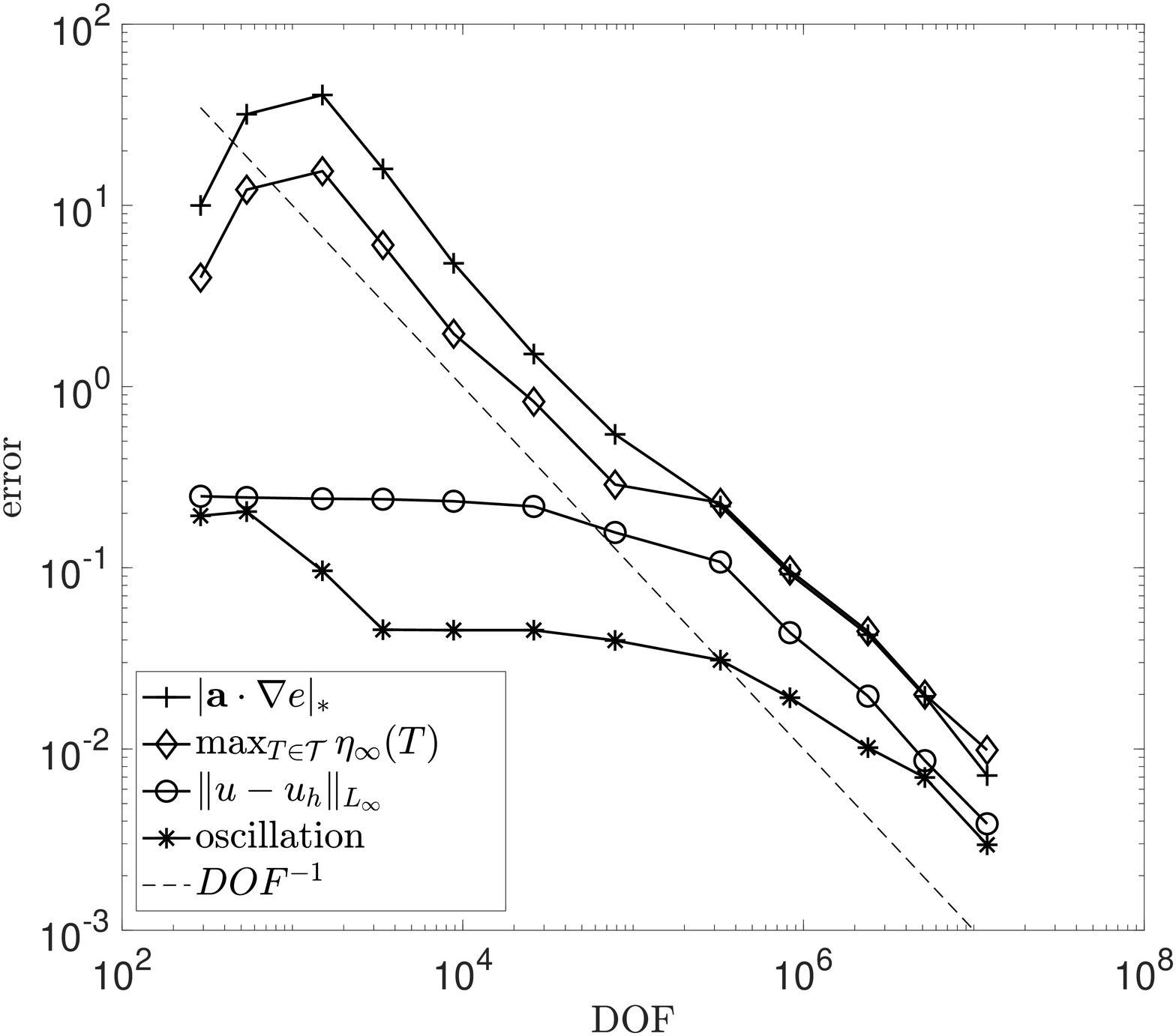}
\caption{{Outflow boundary} layer {solution $u_2$ with}  $\epsilon=10^{-5}$  and streamline diffusion stabilization.  Uniform and adaptive refinement (left), and with oscillation displayed (right).}
\label{fig3}
\end{center}
\end{figure}

\subsection{Experiment 3: {Interior layer,} $|\cdot|_*$ versus $|\cdot|_{**}$ error measure}
Here, in addition to the {outflow boundary} layer problem considered above,
{we also consider the simple interior layer solution \cite{Ern_Steph_2008}}
$$
u_3(x,y)=2x(1-x)y(1-y)\left ( 1-\tanh[ (.5-x)/\sqrt{\epsilon}] \right ).
%\color{yellow} ~{\bf a}=[0,1]
$$
Note that the interior layer solution possesses milder layer behavior than does the outflow layer.

{We shall now use these two tests---see Figure \ref{fig4}---not only to illustrate the performance of our estimator, but
to also
follow up on the theoretical comparison of
$|\cdot|_*$ versus $|\cdot|_{**}$ in Section~\ref{ssec_norms}.}
Recall the definitions \eqref{star_def} and \eqref{starstar_def} of the $\ast$ and $\ast \ast$-(semi)norms, and that the $\ast$-seminorm accurately reflects estimator behavior but may not be intuitive while the $\ast\ast$ norm has a more concrete form but may overestimate the $\ast$-seminorm.   In some cases of interest the $\ast$ and $\ast \ast$ (semi) norms of the convective error are nonetheless very close in size.  To illustrate this fact and the inequality \eqref{norm_compare} we compare behavior of the outflow boundary layer problem and the simple interior layer solution.   In the left plot in Figure \ref{fig4} we take $\epsilon=10^{-4}$ in the {outflow boundary} layer problem.  This yields $|\va\cdot\nabla(u-u_h)|_{\ast \ast} \sim \epsilon^{-1}$, while $|\va\cdot\nabla(u-u_h)|_\ast$ is much smaller.  We see that the quantity $\max_{T \in \T} \alpha_T \|P^0(\va \cdot \nabla )\|_{\infty \,; T}$ closely tracks $|\va \cdot \nabla e|_{\ast}$, while $\max_{T \in \T} \alpha_T \|\va \cdot \nabla e-P^0 (\va \cdot \nabla e)\|_{\infty \,;T}$ closely tracks $|\va \cdot \nabla e|_{\ast \ast}$; cf. \eqref{norm_compare}.  Also, heavy refinement is required before oscillation in $\nabla(u-u_h)$ is resolved and the $\ast$ and $\ast \ast$ quantities become more or less equivalent.  In the right plot we consider the interior layer problem with $\epsilon=10^{-5}$.  Here the $\ast$ and $\ast \ast$ (semi)norms of the convective error are essentially equivalent throughout the convergence history.  These experiments confirm that while the $\ast \ast$ seminorm lends some intuition to the error behavior, the $\ast$ norm most accurately captures the error dynamics.

%%
%The $|\cdot|_\ast$ dual seminorm can be somewhat unyieldy to work with, as it involves taking a maximum of an integral over polynomials and integration over subelements.  The $\ast \ast$ seminorm is on the other hand somewhat more concrete, but as noted above the estimator only reliably bounds $|\nabla(u-u_h)|_{\ast \ast}$ up to oscillation terms.  In fact, $|\va \cdot \nabla e |_{\ast \ast}$ may not even be bounded if $\nabla u$ is unbounded, as may for example happen in the case of corner singularities, while $|\nabla (u-u_h)|_\ast$ will remain finite as long as $\nabla u \in L_1$.

We also emphasize that in all of our experiments, the estimator $\max_{T \in \T} \eta_\infty(T)$ in fact closely tracked the total error $\|u-u_h\|_{L_\infty(\Omega)} + |\va\cdot\nabla(u-u_h)|_{\ast}$ as predicted.  Elementwise oscillation of the residual $R_h$ appeared not to dominate the other terms in the total error in all cases.

\setlength{\unitlength}{.8cm}
\begin{figure}[t]
\begin{center}
\includegraphics[width=2.8in]{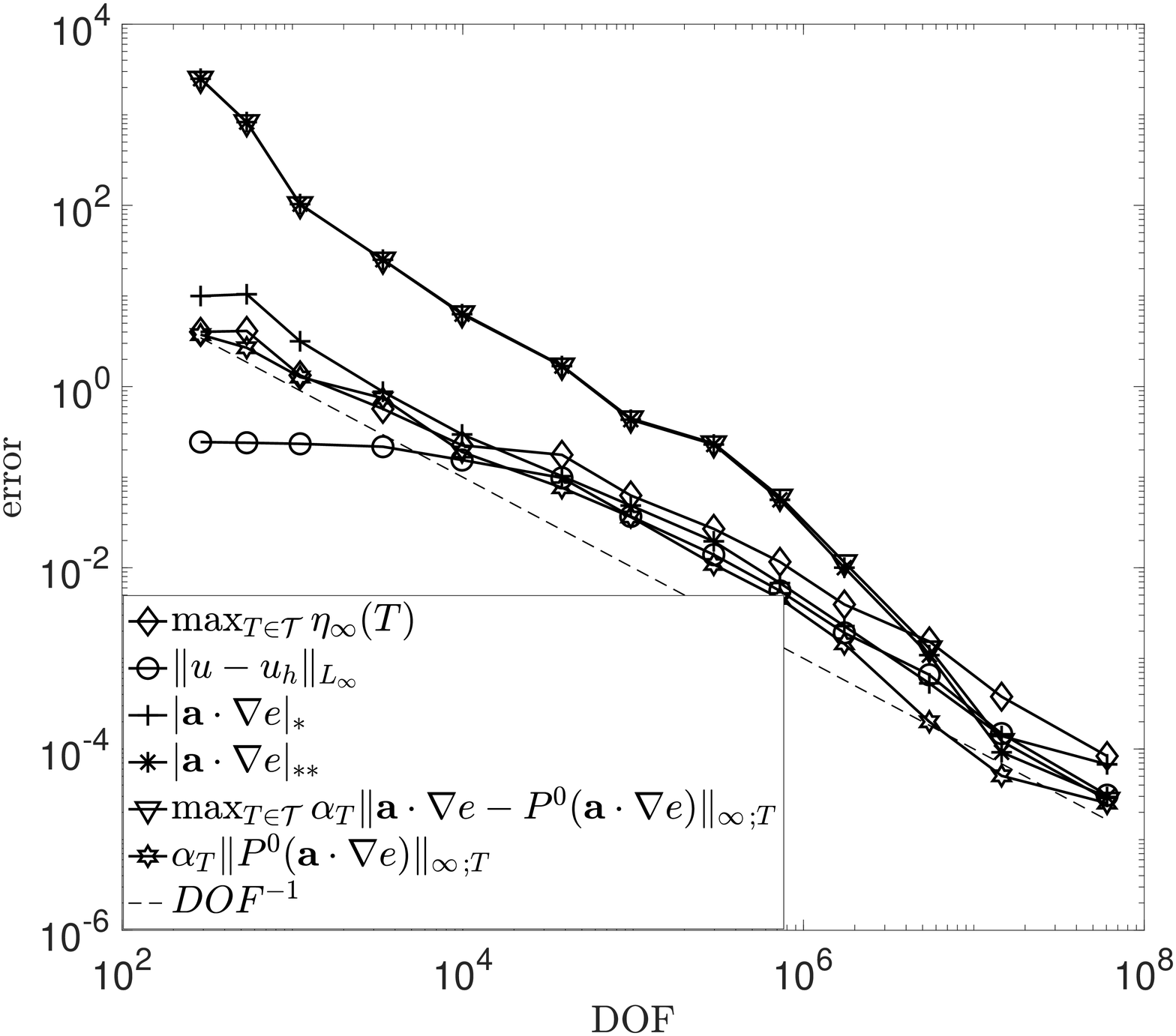}
\includegraphics[width=2.8in]{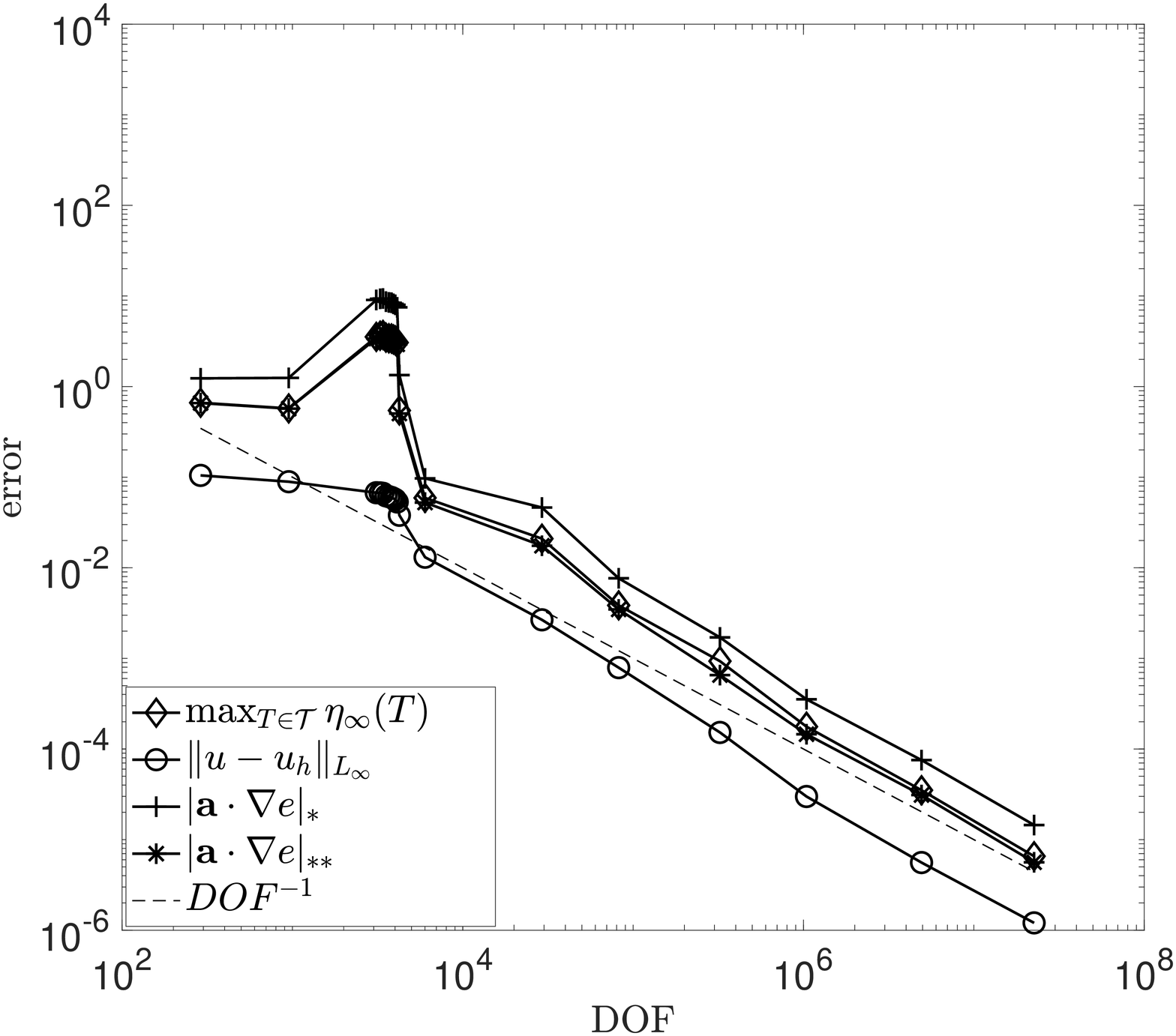}
\caption{{Outflow boundary} layer solution {$u_2$} with $\epsilon=10^{-4}$ (left); interior layer solution
{$u_3$} with $\epsilon=10^{-5}$ (right),
{streamline diffusion stabilization} with  adaptive refinement.}
\label{fig4}
\end{center}
\end{figure}

\section*{Funding}
The first author was partially supported by NSF Grants DMS-1720369 and DMS-2012326.  The third author was partially supported by  Science Foundation Ireland under Grant number 18/CRT/6049.

\bibliographystyle{plain}
\bibliography{data/dk_convection}

\end{document}